\documentclass[10pt, a4paper, english]{article}
\usepackage{preamble}
\begin{document}

\title{Cardinalities in Height 1}
\author{Yifan Li}
\maketitle

\begin{abstract}
\pagenumbering{roman}
In this article, we give an introduction to the notion of $ambidexterity$ and norm map, and construct inductively the canonical norm map for $m$-truncated maps for some $m\geq-1$, on which the definitions of $integration$ and $cardinality$ are built. We then use several propositions to justify the properties of cardinality and integration and their compatibility with monoidal structure. We give a brief introduction of the definition and behaviors of $semiadditive\ height$. Focusing on stable monoidal $p$-local $\infty$-categories of height 1, for any finite group $G$, with the help of \mobius function and Burnside ring, we give an explicit decomposition of the cardinality of $BG$ into an expression of the cardinality of $BC_p$. Eventually, we generalize the result and conclude with a formula of the cardinality of any $\pi$-finite space $A$.
\end{abstract}
\newpage

\tableofcontents{}
\newpage
\pagenumbering{arabic}
\section{Introduction}

\subsection{Background}
The notion of \emph{Ambidexterity} was introduced in \cite{AmbiKn} for the study of $K(n)$-$local$ spectra. Informally speaking, a space $A$ is \catC-ambidextrous, where \catC\ is an \inftycat\  means that there is a way to sum up, or ``integrate'' $A$-family of maps from $c$ to $d$, where $c$ and $d$ are fixed objects in \catC. An \inftycat\  \catC\  is called $m$-$semiadditive$ for some $m\geq-1$ if any $m$-finite space is \catC-ambidextrous, and $\infty$-$semiadditive$ if it is $m$-semiadditive for all $m\geq -1$.

\cite{ambiSpan} continued the study of $m$-semiadditivity and proved the universality of the \inftycat\  of the spans of $m$-finite spaces among $m$-semiadditive \inftycats.

\cite{TeleAmbi} followed the work of \cite{AmbiKn}, and went further to telescopic local spectra.

\cite{AmbiHeight} used ambidexterity to give a definition of $semiadditive\ height$, which generalizes the chromatic height in certain cases. With this fundamental invariant defined, we have sufficient conditions for a stable $m$-semiadditive \inftycat\ to be $\infty$-semiadditive, and there are more interesting phenomena discovered for semiadditive height, including so-called ``semiadditive redshift'' with regard to categorification.

Before introducing the main result and the organization of this thesis, we now give an intuitive explanation of the notion of ambidexterity and higher semiadditivity, so that we could have an intuitive taste of this notion and the results of the papers above.

\subsubsection{Ambidexterity and Integration}
Let us for now consider ordinary categories. Recall the definition of semiadditive category:

\begin{defn}
Category \catC \ is semiadditive if:
\begin{itemize}
    \item \catC \ admits all, including empty, finite products and coproducts, including the empty case, with initial object $\varnothing$ and final object $\ast$
    \item for $c_1, c_2, ..., c_n\in Obj($\catC$)$, there exists natural isomorphism $\coprod_{1\leq i \leq n}c_i \to \prod_{1\leq i \leq n}c_i$. In particular, the unique map $0\to 1$ is an isomorphism, which means that \catC \ is pointed.
\end{itemize}
\end{defn}

\begin{rem}
We have a canonical candidate for the natural isomorphism in the above definition, namely the one given by the ``identity'' matrix, characterized by the following condition:
\begin{equation*}
    c_i \to \coprod_{1\leq i \leq n}c_i \to \prod_{1\leq i \leq n}c_i \to c_j = \left \{
\begin{aligned}
    \Id_{c_i}\text{ if } i = j,\\
    0_{i,j} \text{ if }i\neq j.
\end{aligned}
\right.
\end{equation*}
In fact, for \catC \ to be semiadditive, it suffices to check that the above canonical morphism just defined is an isomorphism for any $c_1, c_2, ..., c_n \in obj($\catC$)$.
\end{rem}

Being semiadditive is a property, and the main feature is that we could sum up a finite family of morphisms between two objects. For \inftycat \catC, which is enriched in topological spaces, the motivation is to generalize to get an analogous feature, so that we could sum up, or integrate a family of morphisms indexed by a space $A$ between two objects. Namely, given a space $A$ and a map
\begin{equation*}
    \phi: A \to \Map_{\mathcal{C}}(X, Y),
\end{equation*}
we want to define a map
\begin{equation*}
    \int_A \phi: X \to Y.
\end{equation*}

The construction is straightforward as we can follow the construction of finite summation of semiadditive category:
\begin{equation*}
    X \oto{\Delta} \limit{A}X \oto{\limit{A}\phi} \limit{A}Y \iso \colim{A}Y \oto{\nabla} Y,
\end{equation*}
but this requires several assumptions. Obviously, we need to assume that the appeared limits and colimits do exist, moreover, the key assumption is that there is a canonical isomorphism $\limit{A}Y \iso \colim{A}Y$, and we shall seek an inductive construction for such isomorphism.

In general, for a diagram $F: A \to \mathcal{C}$, a morphism $\eta: \colim{A}F \to \limit{A}F$ is equivalent to specifying a compatible family of morphisms
\begin{equation*}
    a, b\in A,   \eta_{a,b}: F(a)\to F(b).
\end{equation*}

Now for $a, b\in A$ fixed, and denote the space of paths from $a$ to $b$ by $A_{a, b}$, the diagram $F$ provides a diagram:\begin{equation*}
    F_{a,b} : A_{a,b} \to Map(F(a), F(b)),
\end{equation*}
and unlike the case of semiadditivity, there is a priori no canonical compatible way to choose one morphism or path from the diagram. However, by assuming that we could canonically sum up the $A_{a,b}$-indexed family of maps, then we can consider the integration
\begin{equation*}
    \eta_{a,b}=\int_{A_{a,b}}F_{a,b}: F(a)\to F(b),
\end{equation*}
then this choice is compatible, and we could equivalently construct the following map:
\begin{equation*}
    \Nm_A: \colim{A}F \to \limit{A}F,
\end{equation*}
it is called the norm map.

The above construction is in general for any diagram $F: A\to \mathcal{C}$, but in order to define the integration, we just need to consider the special case where $F$ is constant on an object $Y$. A morphism $\colim{A}Y \to \limit{A}Y$ is the same as a map of spaces
\begin{equation*}
    A \times A \to \Map_{\mathcal{C}}(Y,Y).
\end{equation*}
And for $(a,b)\in A\times A$, the $\eta_{a,b}$ constructed above is $\int_{A_{a,b}}\Id_Y$, which we denote by $\card{A_{a,b}}_Y$, and if the norm map
\begin{equation*}
    \Nm_A^Y: \colim{A}Y \to \limit{A}Y
\end{equation*}
is an isomorphism for any $Y\in \mathcal{C}$, we are finally able to sum up any family of morphisms indexed by $A$, in which case we say $A$ is \catC-ambidextrous. In fact, we prove in Lemma \ref{check constant} that if $\Nm_A^Y$ is an isomorphism for any $Y\in\mathcal{C}$, then $\Nm_A:\colim{A}F\to\limit{A}F$ is an isomorphism for general diagram $F$.

\subsubsection{Inductive Process}
As we have seen, the construction of the norm map of $A$ relies recursively on the norm map of $A_{a,b}$, therefore we can give a general construction of the norm map of $m$-finite spaces, inductively on $m$.

\begin{defn}\label{m-finite}
For $m\geq 0$, space $A$ is $m$-finite if:
\begin{itemize}
    \item $\pi_n(A)=0$ for $n>m$;
    \item $\card{\pi_0(A)}<\infty$ and $\card{\pi_n(A,a)}<\infty$ for any $n\geq 1$ and $a\in A$.
\end{itemize}
\end{defn}

We have two special cases in particular, namely, $A$ is $(-2)$-finite, if it is contractible and $(-1)$-finite, if it is contractible or empty space.

Note that it is a general fact that if $A$ is $m$-finite, $m>-2$, then its path spaces $A_{a,b}$ are $(m-1)$-finite. Therefore if we assume that for \catC, all $(m-1)$-finite spaces are \catC-ambidextrous, then we define the norm maps as above, and it is a $property$ of \catC\  that the norm maps for all $m$-finite spaces are isomorphisms, which inductively helps to define the norm maps for $(m+1)$-spaces. The property of \catC\  that all $m$-finite spaces are \catC-ambidextrous is called $m$-semiadditive, which is a generalization of the property of semiadditivity.

Let us see some low-dimensional examples:
\begin{itemize}
    \item[-2:] $(-2)$-finite spaces are contractible, the canonical norm maps are just the identity. The correspondent integration of one point is just taking the value of itself.
    \item[-1:] $(-1)$-finite spaces are either contractible or empty. The former case is treated as above. In the  latter case, the colimit is always the initial object $\varnothing$, while the limit is always the final object $\ast$. The norm maps are just the unique morphism $\varnothing\to\ast$, therefore to say \catC\ is $(-1)$-semiadditive is just saying that \catC\ is pointed.
    \item[0:] This is just the case of ordinary semiadditivity, where we start with. It is worth noticing that in the construction of norm maps
    \begin{equation*}
        \coprod_{1\leq i\leq n} c_i \to \prod_{1\leq i\leq n} c_i,
    \end{equation*}
    the use of ``identity matrix'' is exactly the inductive step of using the $(-1)$-finite integration.
    \item[1:] In this case, we may assume that $A$ is connected as then the unconnected case follows. A connected $1$-finite space is of the form $BG$ where $G$ is a finite group, and a diagram $F: BG\to \mathcal{C}$ is equivalent to a homotopy $G$-action on an object $X$ of \catC. The limit and colimit of the diagram are the homotopy invariant and coinvariant respectively, and if \catC\ is semiadditive, we have the norm map
    \begin{equation*}
        \Nm_{BG}^X: X_{hG} \to X^{hG},
    \end{equation*}
    which can be informally thought of as summing over the orbits.
\end{itemize}

\begin{rem}
In the following chapters, we are going to define norm map, ambidexterity, and $m$-semiadditivity in a more axiomatic way. The informal constructions and examples in this section are mainly for the purpose of having the initial idea and motivation of the higher generalization, which gives prototype examples that help to ``visualize'' the formal constructions and further definitions on these higher structures and properties.
\end{rem}
\newpage
\subsection{Organization}
In this subsection, we give a description of the content and purpose of each of the following sections.
\subsubsection{Section 2}
In this section, we partly follow the axiomatic framework of normed functors and integration developed in \cite{TeleAmbi} and apply the general results to our main focus, the local systems. We study various functorial properties of integration and its compatibility with the symmetric monoidal structure. We define the cardinality of a map as a special case of integration, and give propositions that help with the calculation of the cardinality.
\subsubsection{Section 3}
With the framework built in the previous section, we focus on $p$-local \inftycat\ and build the definition of semiadditive height in terms of cardinalities of certain Eilenberg-Maclane spaces. We state several main theorems in \cite{AmbiHeight}, and briefly introduce the theory of modes.
\subsubsection{Section 4}
In this section, we introduce our main algebraic tool for calculation, namely the Burnside ring and the \mobius function. We define the \mobius function in general for locally finite posets, with the classical example for integers. Then we study the Burnside ring and give a formula for its primitive idempotents with \mobius function, following the setting and method of \cite{BurnsideIdem}.
\subsubsection{Section 5}
We develop the main result of this thesis in this section. For $p$-local symmetric monoidal \inftycat\ \catC\ and space $BG$, we give a ring morphism from the Burnside ring of $G$ to the endomorphism space of the unit of \catC that encodes the cardinalities of the classifying spaces of subgroups of $G$. Then with the formula from the previous section, we can express the cardinality of $BG$ in terms of the cardinalities of classifying spaces of $p$-subgroups of $G$. In the case of height 1, we can further improve to an expression with just the cardinality of $BC_p$. Finally, we will generalize to spaces with higher homotopy structure, and finish with a formula for the cardinality of any $\pi$-finite space.

For the composition and structure of the thesis, we try our best to make it self-contained. However, there are many more great properties and insights that cannot be covered thoroughly in this thesis. Therefore, while we try to give the full proof or a description of the idea of proof, we are always willing to state and refer to other interesting and important results without proof alongside, in the belief that this helps to understand the larger picture of the theory.
\newpage
\subsection{Acknowledgments}
This article originated from my master's thesis. The deepest gratitude goes to my supervisor, Shachar Carmeli, for introducing me to the interesting field of ambidexterity and height and for the patient and generous instructions and guidance throughout this project. Working with homotopy could be confusing occasionally, but the discussions with him always inspire me and solve the problems. 


\newpage
\subsection{Terminology and Notation}
In this thesis, we work on the notations developed in \cite{htt} and \cite{HA}. For the theory of ambidexterity and height, we follow the notations used consistently in \cite{TeleAmbi} and \cite{AmbiHeight}. On the algebraic side, we use notations slightly modified from that of \cite{BurnsideIdem}.
\begin{enumerate}\addtocounter{enumi}{-1}
    \item As we are working with \inftycats, it is necessary to note that the terms are used in the correspondent sense. For example, when we say isomorphism, we mean an invertible morphism, i.e. an $equivalence$. Similarly, when we consider commutative diagrams, we think of them as homotopically commutative diagrams. Also, the pullbacks, pushforwards, and other limits and colimits of spaces are taken in the \inftycat\ $\mathcal{S}$.
    \item For functor $F: \mathcal{C}\to \mathcal{D}$, a left adjoint to $F$ is a pair $(L,u)$, with functor $L: \mathcal{D}\to\mathcal{C}$, and unit natural transformation $u: \Id_{\mathcal{D}}\to F \circ L$ as defined in \cite[Definition 5.2.2.7]{htt}.
    \item A right adjoint to $F$ is dually defined as pair $(R,c)$, with functor $R:\mathcal{D}\to\mathcal{C}$, and a counit natural transformation $c: F\circ R \to \Id_D$ defined as the dual of \cite[Definition 5.2.2.7]{htt}.
    \item For a map of spaces $q:A\to B$ and \inftycat\ \catC, $q$-limits (colimits) are the limits (colimits) indexed by the (homotopy) fibers of $q$.
    \item The definition of $m$-finite space is given in Definition \ref{m-finite}. For a map of spaces $q:A\to B$, $q$ is $m$-finite, $m\geq -2$, if the fibers of $q$ are all $m$-finite.
    \item A space or map of spaces is $\pi$-finite if it is $m$-finite for some $m\geq -2$.
    \item For a group $G$, we use $H\leq G$ to denote a subgroup $H$ of $G$.
\end{enumerate}

\newpage

\section{Norm and Ambidexterity}

\subsection{Normed Functors and Integration}
\subsubsection{Normed Functor}
We begin by reviewing the definition and basic facts of normed functors as in \cite{TeleAmbi}.

\begin{defn}
For \inftycats\ \catC\ and \catD, a normed functor
\begin{equation*}
    q:\mathcal{D}\normed \mathcal{C}
\end{equation*}
consists of functor $q^*:\mathcal{C}\to\mathcal{D}$, its left adjoint $q_!:\mathcal{D}\to\mathcal{C}$ unit $u_!^q$, right adjoint $q_*:\mathcal{D}\to\mathcal{C}$ with counit $c_*^q$, and a natural transformation
\begin{equation*}
    \Nm_q: q_!\to q_*
\end{equation*}
which we call a $norm$. If $\Nm_q$ is moreover an natural isomorphism natural transformation, $q$ is $iso$-$normed$.
\end{defn}

\begin{rem}
We shall sometimes refer directly to $q^*$ as a normed functor. As the right and left adjoints are essentially unique when they exist, this convention should not be ambiguous.
Also, we shall use $c_!^q$ and $u_*^q$ to denote the correspondent counit and unit map of $u_!^q$ and $c_*^q$ respectively. For simplicity, we shall drop the $q$-subscript when it is clear that the normed functor is $q$.
\end{rem}

\begin{note}
We have another way to demonstrate the norm map, namely, a norm $\Nm_q:q_!\to q_*$ is equivalent to a natural transformation $\nu_q: q^*q_!\to \Id$. If $q$ is iso-normed, then $\nu_q$ is a counit, exhibiting $q_!$ as a right adjoint of $q^*$. We therefore call $\nu_q$ the "wrong way counit" even when $q$ is not iso-normed.
\end{note}

The next lemma will be useful to justify the later definition of ambidexterity of spaces.

\begin{lem}\label{check constant}
For a normed functor $q:\mathcal{D}\normed\mathcal{C}$, if the norm $\Nm_q:q_!\to q_*$ is an isomorphism at $q^*X$ for all $X\in\mathcal{C}$, then $q$ is iso-normed.
\end{lem}
\begin{proof}
Consider the following diagrams:
\begin{equation*}
    \begin{tikzcd}
        &q_!q^*q_*\arrow{r}{c_*}\arrow{d}{\Nm_q}[swap]{\sim}&q_!\arrow{d}{\Nm_q}\\
        q_*\arrow{r}{u_*}&q_*q^*q_*\arrow{r}{c_*}&q_*,
        \end{tikzcd}
    \qquad
    \begin{tikzcd}
        q_!\arrow{r}{u_!}\arrow{d}{\Nm_q}&q_!q^*q_!\arrow{r}{c_!}\arrow{d}{\Nm_q}[swap]{\sim}&q_!\\
        q_*\arrow{r}{u_!}&q_*q^*q_!.
    \end{tikzcd}
\end{equation*}
The diagrams commute from the naturality of (co)unit maps, and by the zig-zag identities, the composition along the bottom row of the left diagram is the identity. Since the left vertical arrow is an isomorphism, the left diagram shows that $\Nm_q$ has a left inverse. Similarly, the right diagram shows that $\Nm_q$ has a right inverse so $\Nm_q$ is an isomorphism.
\end{proof}

A trivial example of normed functors is just the identity normed functor $\Id:\mathcal{C}\to\mathcal{C}$, with $q^*=q_!=q_*=\Id$, and the norm is just the identity transformation. As adjunctions can be composed, we can thereby easily define the composition of normed functors.

\begin{defn}
Given a pair of functors that is composable
\begin{equation*}
    \mathcal{E}\onormed{p}\mathcal{D}\onormed{q}\mathcal{C},
\end{equation*}
we have the composition $qp:\mathcal{E}\normed\mathcal{C}$, with
\begin{equation*}
    (qp)^*=p^*q^*,\ (qp)_!=q_!p_!,\ (qp)_*=q_*p_*
\end{equation*}
and the composition of norms:
\begin{equation*}
    \Nm_{qp}:q_!p_!\oto{\Nm_q}q_*p_!\oto{\Nm_q}q_*p_*,
\end{equation*}
or equivalently the composition of wrong way counit
\begin{equation*}
    \nu_{qp}:p^*q^*q_!p_!\oto{\nu_q}p^*p_!\oto{\nu_p}\Id.
\end{equation*}
It is obvious that if $p$ and $q$ are iso-normed, so is $qp$.
\end{defn}

\subsubsection{Integration}
With an iso-normed functor, we can associate an integration operation of a map as follows.

\begin{defn}\label{int}
Let $q:\mathcal{D}\normed\mathcal{C}$ be an iso-normed functor, for each $X, Y\in \mathcal{C}$, we define the $integral$ map
\begin{equation*}
    \int_q:\Map_{\mathcal{D}}(q^*X, q^*Y)\to \Map_{\mathcal{C}}(X,Y)
\end{equation*}
to be the following composition
\begin{equation*}
\Map_{\mathcal{D}}(q^*X,q^*Y)\oto{q_*}\Map_{\mathcal{C}}(q_*q^*X,q_*q^*Y)\oto{\Nm_q^{-1}}\Map_{\mathcal{C}}(q_*q^*X,q_!q^*Y)\oto{c_!\circ-\circ u_*}\Map_{\mathcal{C}}(X,Y).
\end{equation*}
or equivalently, using the correspondent unit $\mu_q:\Id\to q_!q^*$ of the wrong way counit, as the composition
\begin{equation*}
\Map_{\mathcal{D}}(q^*X,q^*Y)\oto{q_!}\Map_{\mathcal{C}}(q_!q^*X,q_!q^*Y)\oto{c_!\circ - \circ \mu}\Map_{\mathcal{C}}(X,Y).
\end{equation*}
\end{defn}

\begin{rem}
It is obvious and worth noticing that the integral map is natural in $X$ and $Y$.
\end{rem}

In particular, we define the cardinality of the normed functor $q$.

\begin{defn}
Let $q:\mathcal{D}\to\mathcal{C}$ be an iso-normed functor, the cardinality of $q$ is defined as the natural transformation of $\Id_{\mathcal{C}}$
\begin{equation*}
    \card{q}\coloneqq c_!^q\circ \mu_q.
\end{equation*}
In particular, by comparing the two equivalent definitions in Definition \ref{int} it is easy to see that the transformation morphism at $X\in\mathcal{C}$ in $\Map_{\mathcal{C}}(X,X)$ is
\begin{equation*}
    \card{q}_X = \int_q\Id_{q^*X} = \int_q q^*\Id_X.
\end{equation*}
\end{defn}

The integration defined as above follows some nice rules with which we are familiar analogously from the integration in the sense of analysis.

\begin{prop}[homogeneity]\label{homogeneity}
Let $q:\mathcal{D}\to\mathcal{C}$ be an iso-normed functor, and $X,Y,Z\in\mathcal{C}$.
\begin{enumerate}
    \item For all maps $f:q^*X\to q^*Y$ and $g:Y\to Z$,
    \begin{equation*}
        g\circ \left(\int_qf\right) = \int_q \left(q^*g\circ f\right) \in \Map_{\mathcal{C}}(X, Z).
    \end{equation*}
    \item For all maps $f:X\to Y$ and $g:q^*Y\to q^*Z$,
    \begin{equation*}
        \left(\int_qg\right)\circ f = \int_q\left(g\circ q^*f\right)\in \Map_{\mathcal{C}}(X,Z).
    \end{equation*}
\end{enumerate}
\end{prop}
\begin{proof}
(1) Consider the commutative diagram
\begin{equation*}
    \begin{tikzcd}
        X\arrow{r}{\mu}\arrow{d}{\mu}
        &q_!q^*X\arrow{r}{f}\arrow{d}{f}
        &q_!q^*Y\arrow{r}{c_!}\arrow{d}{g}
        &Y\arrow{d}{g}\\
        q_!q^*X\arrow{r}{f}
        &q_!q^*Y\arrow{r}{g}
        &q_!q^*Z\arrow{r}{c_!}
        &Z.
    \end{tikzcd}
\end{equation*}
The composition along the top row and then the right
column is $g\circ \left(\int_qf\right)$, while the composition along the left column and then the bottom row is $\int_q \left(q^*g\circ f\right)$.

(2) Similar as in (1), but using the commutative diagram
\begin{equation*}
    \begin{tikzcd}
        X\arrow{r}{\mu}\arrow{d}{f}
        &q_!q^*X\arrow{r}{f}\arrow{d}{f}
        &q_!q^*Y\arrow{r}{g}\arrow{d}{g}
        &q_!q^*Z\arrow{d}{c_!}\\
        Y\arrow{r}{\mu}
        &q_!q^*Y\arrow{r}{g}
        &q_!q^*Z\arrow{r}{c_!}
        &Z.
    \end{tikzcd}
\end{equation*}
\end{proof}

Also, given a pair of composable iso-normed functors $\mathcal{E}\onormed{p}\mathcal{D}\onormed{q}\mathcal{C}$, an analogy of Fubini's Theorem is satisfied.

\begin{prop}[Higher Fubini's Theorem]\label{Fubini}
For composable normed functors as above, $X,Y\in\mathcal{C}$, and $f:p^*q^*X\to p^*q^*Y$,
\begin{equation*}
    \int_q\left(\int_pf\right)=\int_{qp}f\in \Map_{\mathcal{C}}(X, Y).
\end{equation*}
\end{prop}
\begin{proof}
Consider the following diagram
\begin{equation*}
    \begin{tikzcd}
        &q_*p_*p^*q^*X\arrow{r}{f}
        &q_*p_*p^*q^*Y\arrow{r}{\Nm_{qp}}\arrow[swap]{d}{\Nm_p^{-1}}
        &q_!p_!p^*q^*Y\arrow{rd}{c_!^{qp}}\arrow[equal]{d}\\
        X\arrow{ur}{u_*^{qp}}\arrow{dr}{u_*^q}
        &&q_*p_!p^*q^*Y\arrow{r}{\Nm_q^{-1}}\arrow[swap]{d}{c_!^p}
        &q_!p_!p^*q^*Y\arrow[swap]{d}{c_!^p}
        &Y.\\
        &q_*q^*X\arrow{r}{\int_p f}\arrow{uu}{u_*^p}
        &q_*q^*Y\arrow{r}{\Nm_q^{-1}}
        &q_!q^*Y\arrow{ur}{c_!^q}
    \end{tikzcd}
\end{equation*}
The triangles and top right square commute by the definition of composition of normed functors, the bottom right square commutes by the naturality of counit, and the left square commute by the definition of $\int_pf$. The composition along the top path from $X$ to $Y$ is $\int_{qp}f$, while the composition along the bottom path is $\int_q\left(\int_pf\right)$, therefore they are isomorphic.
\end{proof}

\subsubsection{Beck-Chevalley Conditions and Ambidextrous Squares}

The definitions in this section are in form of commutative diagrams, which seem a bit complicated at first glance. However, they lead to the simple and important result at the end of the section. First, let us look at a commutative square involving adjoint functors.

\begin{defn}
As defined in \cite{htt}, a commutative diagram of functors
\begin{equation*}\tag{$\square$}
    \begin{tikzcd}
    \mathcal{C} \arrow{r}{F_{\mathcal{C}}} \arrow[swap]{d}{q^*} & \title{\mathcal{C}} \arrow{d}{\tilde{q}^*} \\%
\mathcal{D} \arrow{r}{F_{\mathcal{D}}}& \tilde{\mathcal{D}}
    \end{tikzcd}
\end{equation*}
is equivalently a natural isomorphism $F_{\mathcal{D}}q^*\iso \tilde{q}^*F_{\mathcal{C}}$. If the vertical functors admit left adjoints $q_!\leftadj q^*$ and $\tilde{q}_!\leftadj \tilde{q}^*$, we define the $\bc_!$ (Beck-Chevalley) natural transformation
\begin{equation*}
    \beta_!:\tilde{q}_!F_{\mathcal{D}}\oto{\mu_!^q}\tilde{q}_!F_{\mathcal{D}}q^*q_!\iso \tilde{q}_!\tilde{q}^*F_{\mathcal{C}}q_!\oto{c^{\tilde{q}}_!}F_{\mathcal{C}}q_!.
\end{equation*}
Dually, if the vertical functors admit right adjoints $q^*\leftadj q_*$ and $\tilde{q}^*\leftadj \tilde{q}_*$, we define the $\bc_*$ natural transformation
\begin{equation*}
    \beta_*:F_{\mathcal{C}}q_*\oto{u_*^{\tilde{q}}}\tilde{q}_*\tilde{q}^*F_{\mathcal{C}}q_*\iso\tilde{q}_*F_{\mathcal{D}}q^*q_*\oto{c_*^q}\tilde{q}_*F_{\mathcal{D}}.
\end{equation*}
\end{defn}

With the two natural transformations, we can define the following $property$ of $\square$.

\begin{defn}
Consider the same commutative square $\square$, we say it satisfies the $\bc_!$ (resp. $\bc_*$) condition, if $q$ and $\tilde{q}$ both admit left adjoints (resp. right adjoints), and the natural transformation $\beta_!$ (resp. $\beta_*$) is an isomorphism.
\end{defn}

$\beta_!$ and $\beta_*$ are compatible with units and counits of the adjoints, in the sense that the following diagrams are commutative.

\begin{lem}\label{BCtriangle}
\begin{enumerate}
    \item When $q^*$ and $\qtil^*$ admit left adjoints, the following two diagrams are commutative.
    \begin{equation*}
\begin{tikzcd}
    &\qtil^*\qtil_!F_{\mathcal{D}}\arrow{d}{\beta_!}\\
    F_{\mathcal{D}}\arrow{ur}{u_!^{\tilde{q}}}\arrow[swap]{rd}{u_!^q}&\qtil^*F_{\mathcal{C}}q_!\arrow{d}{\wr}\\
    &F_{\mathcal{D}}q^*q_!
\end{tikzcd}
\qquad\qquad\qquad
\begin{tikzcd}
    \qtil_!\qtil^*F_{\mathcal{C}}\arrow{d}{\wr}\arrow{dr}{c_!^{\qtil}}\\
    \qtil_!F_{\mathcal{D}}q^*\arrow{d}{\beta_!}&F_{\mathcal{C}}\\
    F_{\mathcal{C}}q_!q^*\arrow[swap]{ru}{c_!^q}
\end{tikzcd}
    \end{equation*}
    \item Dually, when $q^*$ and $\qtil^*$ admit right adjoints, the following two diagrams are commutative.
    \begin{equation*}
        \begin{tikzcd}
            &F_{\mathcal{C}}q^*q_*\arrow{d}{\beta_*}\\
            F_{\mathcal{C}}\arrow{ur}{u_*^q}\arrow[swap]{dr}{u_*^{\qtil}}&\qtil_*F_{\mathcal{D}}q^*\arrow{d}{\wr}\\
            &\qtil_*\qtil^*F_{\mathcal{C}}
        \end{tikzcd}
\qquad\qquad\qquad
        \begin{tikzcd}
            F_{\mathcal{D}}q^*q_*\arrow{d}{\wr}\arrow{dr}{c_*^q}\\
            \qtil^*F_{\mathcal{C}}q_*\arrow{d}{\beta_*}&F_{\mathcal{D}}\\
            \qtil^*\qtil_*F_{\mathcal{D}}\arrow[swap]{ru}{c_*^{\qtil}}
        \end{tikzcd}
    \end{equation*}
\end{enumerate}
\end{lem}
\begin{proof}
Direct from the definition of $\beta_*$ and $\beta_!$, and the zig-zag identities.
\end{proof}

Now we move on and focus on commutative squares involving normed functors.

\begin{defn}
A $normed\ square$ is a commutative diagram of the form:
\begin{equation*}\tag{$\ast$}
    \begin{tikzcd}
    \mathcal{C}\arrow{r}{F_{\mathcal{C}}}\arrow[swap]{d}{q^*}&\tilde{\mathcal{C}}\arrow{d}{\tilde{q}^*}\\
    \mathcal{D}\arrow{r}{F_{\mathcal{D}}}&\tilde{\mathcal{D}},
    \end{tikzcd}
\end{equation*}
where $q:\mathcal{D}\normed\mathcal{C}$ and $\tilde{q}:\tilde{\mathcal{D}}\normed\tilde{\mathcal{C}}$ are normed functors.
It is $iso$-$normed$ if $q$ and $\tilde{q}$ are iso-normed.
\end{defn}

Given a normed square $(\ast)$, we have the associated $norm$-$diagram$ of functors, which does not necessarily commute:
\begin{equation*}\tag{$\diamond$}
    \begin{tikzcd}
    F_{\mathcal{C}}q_!\arrow{r}{\Nm_q}&F_{\mathcal{C}}q_*\arrow{d}{\beta_*}\\
    \tilde{q}_!F_{\mathcal{D}}\arrow{u}{\beta_!}\arrow{r}{\Nm_{\tilde{q}}}&\tilde{q}_*F_{\mathcal{D}}.
    \end{tikzcd}
\end{equation*}

The commutativity of $(\diamond)$ is a property of the normed square $(\ast)$, as given in the following definition.

\begin{defn}
A normed square is called $weakly\ ambidextrous$ if the associated norm diagram $\diamond$ commutes. An ambidextrous square is a weakly ambidextrous square that is iso-normed.
\end{defn}

\begin{rem}
It is immediate from the diagram $(\diamond)$ that an ambidextrous square satisfies the $\bc_!$ condition if and only if it satisfies the $\bc_*$ condition.
\end{rem}

\begin{note}
We shall sometimes say directly that a diagram $(\ast)$ is normed which implies that the vertical functors are normed. Similarly, we say $(\ast)$ is (weakly) ambidextrous which means that it is already normed by assumption.
\end{note}

As with previous definitions involving norms, we can define the property of being (weakly) ambidextrous using the wrong-way counit and its associated unit, which is justified by the following lemma.

\begin{lem}
Consider commutative square $(\ast)$, the diagram
\begin{equation*}\tag{$\triangleright$}
\begin{tikzcd}
    F_{\mathcal{D}}q^*q_!\arrow{dr}{\nu_q}\\
    \tilde{q}^*F_{\mathcal{C}}q_!\arrow{u}{\wr}&F_{\mathcal{D}}\\
    \tilde{q}^*\tilde{q}_!F_{\mathcal{D}}\arrow{u}{\beta_!}\arrow{ur}[swap]{\nu_{\tilde{q}}}
\end{tikzcd}
\end{equation*}
commutes if and only if $(\diamond)$ commutes.
\end{lem}
\begin{proof}
For diagram $(\diamond)$, it commutes if $\beta_*\Nm_q\beta_!$ and $\Nm_{\tilde{q}}$ are isomorphic, consider the following diagram
\begin{equation*}
    \begin{tikzcd}
        F_{\mathcal{D}}q^*q_!\arrow{r}{\Nm_q}
        &F_{\mathcal{D}}q^*q_*\arrow{d}{\sim}\arrow{dr}{c_*}\\
        \tilde{q}^*F_{\mathcal{C}}q_!\arrow{u}{\sim}\arrow{r}{\Nm_q}
        &\tilde{q}^*F_{\mathcal{C}}q_*\arrow{d}{\beta_*}&F_{\mathcal{D}},\\
        \tilde{q}^*\tilde{q}_!F_{\mathcal{D}}\arrow{r}{\Nm_{\tilde{q}}}\arrow{u}{\beta_!}
        &\tilde{q}^*\tilde{q}_*F_{\mathcal{D}}\arrow[swap]{ur}{\tilde{c}_*}
    \end{tikzcd}
\end{equation*}
by the zig-zag identities, $(\diamond)$ commutes if and only if the bottom square commutes after post-composing with $\tilde{c}_*$. From Lemma \ref{BCtriangle} the triangle commutes, and the upper square commutes obviously. Note that the composition on top $F_{\mathcal{D}}q^*q_!\to F_{\mathcal{D}}$ is $F_{\mathcal{D}}\nu_q$, and the composition on bottom $\tilde{q}^*\tilde{q}_*F_{\mathcal{D}}\to F_{\mathcal{D}}$ is $\nu_{\tilde{q}}F_{\mathcal{D}}$, thus we have $(\diamond)$ commutes if and only if the two maps $\tilde{q}^*\tilde{q}_*F_{\mathcal{D}}\to F_{\mathcal{D}}$ are isomorphic, that is if and only if $(\triangleright)$ commutes.
\end{proof}

We conclude this section with the interaction between integration and ambidextrous squares, which is the main feature we mentioned at the beginning.

\begin{prop}\label{ambiBC}
Consider again the normed square $(\ast)$, assume that it is ambidextrous and satisfies the $\bc_!$ condition (hence the $\bc_*$ condition). For all $X,Y\in \mathcal{C}$ and $f:q^*X\to q^*Y$, we have
\begin{equation*}
    F_{\mathcal{C}}\left(\int_qf\right)=\int_{\qtil}F_{\mathcal{D}}\left(f\right)\in \Map_{\tilde{\mathcal{C}}}(F_{\mathcal{C}}X,F_{\mathcal{C}}Y).
\end{equation*}
In particular, for all $X\in\mathcal{C}$,
\begin{equation*}
    F_{\mathcal{C}}\left(\card{q}_X\right)=\card{\qtil}_{F_{\mathcal{C}}\left(X\right)}\in \Map_{\tilde{\mathcal{C}}}(F_{\mathcal{C}}X,F_{\mathcal{C}}Y).
\end{equation*}
\end{prop}
\begin{proof}
The second equation is just the first applied to map $f=q^*Id_X$.

For the first equation, construct the following diagram:
\begin{equation*}
    \begin{tikzcd}
        &F_{\mathcal{C}}q_*q^*X\arrow{d}{\beta_*}[swap]{\sim}\arrow{r}{f}
        &F_{\mathcal{C}}q_*q^*Y\arrow{d}{\beta_*}[swap]{\sim}\arrow{r}{\Nm_q^{-1}}
        &F_{\mathcal{C}}q_!q^*Y\arrow{dr}{c_!}\\
        F_{\mathcal{C}}X\arrow{ru}{u_*}\arrow[swap]{dr}{\tilde{u}_*}
        &\tilde{q}_*F_{\mathcal{D}}q^*X\arrow{r}{f}\arrow[dash]{d}{\sim}
        &\tilde{q}_*F_{\mathcal{D}}q^*Y\arrow[dash]{d}{\sim}\arrow{r}{\Nm_{\tilde{q}}^{-1}}
        &\tilde{q}_!F_{\mathcal{D}}q^*Y\arrow[dash]{d}{\sim}\arrow{u}{\beta_!}[swap]{\sim}
        &F_{\mathcal{C}}Y.\\
        &\tilde{q}_*\tilde{q}^*F_{\mathcal{C}}X\arrow{r}{f}
        &\tilde{q}_*\tilde{q}^*F_{\mathcal{C}}Y\arrow{r}{\Nm_{\tilde{q}}^{-1}}
        &\tilde{q}_!\tilde{q}^*F_{\mathcal{C}}Y\arrow[swap]{ur}{\tilde{c_!}}
    \end{tikzcd}
\end{equation*}
The triangles commute by Lemma \ref{BCtriangle}, the top right square commutes by assumption, and the rest squares commute for obvious reasons. The composition along top path gives $F_{\mathcal{C}}\left(\int_qf\right)$ and the composition along bottom path gives $\int_{\qtil}F_{\mathcal{D}}\left(f\right)$, so they are isomorphic.
\end{proof}

\subsubsection{Monoidal Structure}
In this section, we study norms and integration with the existence of (symmetric) monoidal structures on the source and target \inftycats.

As in the previous section, we first define the compatibility condition to work with.

\begin{defn}\label{defn tensor normed}
Let \catC\  and \catD\  be monoidal \inftycats, and consider a normed functor $q:\mathcal{D}\normed\mathcal{C}$. $q$ is called a $\otimes$-$normed\ functor$, if $q^*$ is monoidal (and hence $q_!$ is oplax or colax monoidal), and for all $Y\in\mathcal{D}$ and $X\in\mathcal{C}$, the compositions of the canonical maps
\begin{equation*}
    q_!(Y\otimes(q^*X))\to(q_!Y)\otimes(q_!q^*X)\oto{\Id\otimes c_!}(q_!Y)\otimes X
\end{equation*}
and
\begin{equation*}
    q_!((q^*X)\otimes Y)\to(q_!q^*X)\otimes(q_!Y)\oto{c_!\otimes\Id}X\otimes(q_!Y)
\end{equation*}
are isomorphisms.
\end{defn}

\begin{rem}
Note that the composition of the canonical maps in the definition above does not involve the norm map, and them being isomorphisms is actually a property of $q^*$.
\end{rem}

We conclude this section with a proposition indicating that the interaction of the integral map and monoidal structure is very compatible, and we will explore this interaction more concretely in the next section.

\begin{prop}[{\cite[Proposition 2.3.4]{TeleAmbi}}]\label{general iso-normed}
Let $q:\mathcal{D}\normed\mathcal{C}$ be a $\otimes$-normed functor, the following are equivalent:
\begin{enumerate}
    \item $q$ is iso-normed;
    \item $\Nm_q$ is an isomorphism at $q^*\one_{\mathcal{C}}$.
\end{enumerate}
\end{prop}
\begin{rem}
The Proposition 2.3.4 of \cite{TeleAmbi} actually gives another equivalent statement, which defines a map $\epsilon:\one_q\otimes \one_q\to\one_{\mathcal{C}}$ and says that $\epsilon$ exhibits $\one_q$ as a self dual object in \catC. The notion of dual objects follows from \cite{DualHist} and \cite{TopFieldTheory}. We will not use this interesting perspective though. 
\end{rem}

\subsection{Canonical Norms and Higher Semiadditivity}
We have established the basic formal framework of normed functors, in this section we specify to the case of normed functors between local systems, and define inductively the canonical norms for local systems on truncated spaces and the notion of $m$-semiadditivity. We use the general results from the previous section to develop their corresponding specific version.

\subsubsection{Canonical Norms}
Let \catC\  be an \inftycat, and $A$ a space, regarded as an $\infty$-groupoid. We call $\Fun(A,\mathcal{C})$ the \inftycat\ of \catC-valued $local\ systems$ on $A$. Consider a map of spaces $q:A\to B$, and assume that \catC\ admits all $q$-limits and $q$-colimits.

Now consider the functor given by pre-composition by $q$:
\begin{equation*}
    q^*:\Fun(B,\mathcal{C})\to\Fun(A,\mathcal{C}).
\end{equation*}
As \catC\ admits all $q$-limits and $q$-colimits, $q^*$ have left and right adjoint, given by left and right Kan extension respectively. Our goal in this section is to define the canonical norm map $\Nm_q:q_!\to q_*$. First we introduce the tools for inductive construction.

\begin{defn}[Base Change]\label{base change}
For an \inftycat\ \catC\ and pullback square spaces
\begin{equation*}\tag{$\ast$}
    \begin{tikzcd}
        \tilde{A}\arrow{r}{s_A}\arrow{d}{\tilde{q}}&A\arrow{d}{q}\\
        \tilde{B}\arrow{r}{s_B}&B,
    \end{tikzcd}
\end{equation*}
the associated $base$-$change$ square (valued in \catC) is
\begin{equation*}\tag{$\square$}
    \begin{tikzcd}
        \Fun(B,\mathcal{C})\arrow{r}{s_B^*}\arrow{d}{q^*}&\Fun(\tilde{B},\mathcal{C})\arrow{d}{\tilde{q}^*}\\
        \Fun(A,\mathcal{C})\arrow{r}{s_A^*}&\Fun(\tilde{A},\mathcal{C}).
    \end{tikzcd}
\end{equation*}
\end{defn}

\begin{lem}[{{\cite[Lemma 3.1.2]{TeleAmbi}}}]\label{inductive lemma}
Using the same notation as above. If \catC\  admits all $q$-colimits (resp. $q$-limits), then $\square$ satisfies the $\bc_!$ (resp. $\bc_*$) condition.
\end{lem}

Now for map of spaces $q:A\to B$, consider the diagonal map $\delta:A\to A\times_B A$. It is in general true that if $q$ is $m$-truncated, $m\geq -1$, then $\delta$ is $(m-1)$-truncated. Let \catC\ be an \inftycat\ that admits all $q$-(co)limits and $\delta$-(co)limits, and assume the existence of an iso-norm
\begin{equation*}
    \Nm_{\delta}:\delta_!\to\delta_*,
\end{equation*}
and consider the commutative diagram
\begin{equation*}
    \begin{tikzcd}
        A\arrow{dr}{\delta}
        \arrow[equal, bend left]{drr}
        \arrow[equal, bend right]{ddr}\\
        &A\times_BA\arrow{r}{\pi_1}
        \arrow{d}{\pi_2}&A
        \arrow{d}{q}\\
        &A\arrow{r}{q}&B,
    \end{tikzcd}
\end{equation*}
denote the unit associated to the wrong way counit of $\Nm_{\delta}$ by $\mu_{\delta}:\Id\to\sigma_!\sigma^*$, and by Lemma \ref{inductive lemma}, $\beta_!$ for $\square$ is an isomorphism, therefore we define the $diagonally\ induced$ norm
\begin{equation*}
    \Nm_q^{diag}:q_!\to q_*,
\end{equation*}
or equivalently with the wrong way counit by the following composition:
\begin{equation*}
    \nu_q:q^*q_!\oto{\beta_!^{-1}}(\pi_2)_!\pi_1^*\oto{\mu_{\delta}}(\pi_2)_!\delta_!\delta^*\pi_1^*\iso\Id.
\end{equation*}
\begin{rem}
As explained in \cite[Remark~4.1.9]{AmbiKn}, the appearance of $\mu_{\delta}$ in the above composition can be interpreted that the diagonally induced norm map $\Nm_q^{diag}$ is obtained by integrating the identity map along $\delta$.
\end{rem}
With the above construction, we can inductively define the canonical norm map.
\begin{defn}
For \inftycat\ \catC\ and integer $m\geq 2$, and map $q:A\to B$. Assume that \catC\ admits all $q$-(co)limits, we say $q$ is:
\begin{enumerate}
    \item $(-2)$-$\mathcal{C}$-ambidextrous, if $q$ is $(-2)$-truncated, namely an isomorphism. In this case, the inverse of $q^*$ is both the left and right adjoint, and the canonical norm is just the identity of the inverse, which is clearly an isomorphism.
    \item weakly $m$-$\mathcal{C}$-ambidextrous, $m\geq -1$, if $q$ is $m$-truncated, and the diagonal $A\oto{\delta}A\times_BA$ of $q$ is $(m-1)$-$\mathcal{C}$-ambidextrous. In this case, we define the canonical norm to be the diagonally induced norm $\Nm_q^{diag}$.
    \item $m$-$\mathcal{C}$-ambidextrous, $m\geq -1$, if the canonical norm is an isomorphism.
\end{enumerate}
\end{defn}
\begin{rem}
By Lemma \ref{check constant}, for a weakly $m$-\catC-ambidextrous map $q:A\to B$ to be $m$-\catC-ambidextrous, it suffices to check that the canonical norm is an isomorphism at $q^*X$ for all $X\in\Fun(B,\mathcal{C})$. In particular, for $q: A\to \pt$, it suffices to check that the canonical norm is an isomorphism at all constant functors in $\Fun(A,\mathcal{C})$, which justifies the informal construction in the section of background introduction.
\end{rem}
\begin{note}
A map of spaces $q:A\to B$ is (weakly) $\mathcal{C}$-ambidextrous, if it is (weakly) $m$-$\mathcal{C}$-ambidextrous for some $m\geq -2$.
\end{note}
\begin{rem}
Note that we can define in general that a truncated map $q:A\to B$ is weakly $\mathcal{C}$-ambidextrous if the diagonal map $\delta$ is $\mathcal{C}$-ambidextrous. However, we believe that it is better to keep track of the truncation and give a clear emphasis on the index $m$ in the definition.
\end{rem}
\begin{defn}
When $q:A\to B$ is weakly $\mathcal{C}$-ambidextrous, we define the associated $canonical$ $normed\ functor$
\begin{equation*}
    q_{\mathcal{C}}^{\can}:\Fun(A,\mathcal{C})\normed\Fun(B,\mathcal{C})
\end{equation*}
with \begin{equation*}
    (q_{\mathcal{C}}^{\can})^*=q^*,\quad (q_{\mathcal{C}}^{\can})_!=q_!,\quad
    (q_{\mathcal{C}}^{\can})_*=q_*,
\end{equation*}
and the norm $\Nm_{q_{\mathcal{C}}^{\can}}:q_!\to q_*$ to be the canonical norm.
\end{defn}
\begin{note}
When \catC\ is understood, we use simplified notation $q^{\can}$ for $q_{\mathcal{C}}^{\can}$, and $\Nm_q$ for $\Nm_{q_{\mathcal{C}}^{\can}}$. Also, we write $\int_q$ and $\card{q}$ instead of $\int_{q^{can}}$ and $\card{q^{can}}$. For $q:A\to \pt$, we say $A$ is \catC-ambidextrous if $q$ is \catC-ambidextrous, and we write $\int_A$ and $\card{A}$ instead of $int_q$ and $\card{q}$.
\end{note}
In light of \cite[Proposition~3.1.9]{TeleAmbi}, the canonical norms are compatible with identity, composition, and base change of maps of spaces. We skip the discussion of isomorphism (identity) as it is included in the definition of $(-2)$-ambidextrous.

\begin{prop}[{\cite[Proposition 3.1.9]{TeleAmbi}}]\label{ambisquare}
Let \catC\ be an \inftycat.
\begin{enumerate}
    \item (Composition) Given (weakly) \catC-ambidextrous maps $q:A\to B$ and $p:B\to C$, the composition $pg:A\to C$ is (weakly) \catC-ambidextrous, and $(pq)^{\can}$ is identified with $p^{\can}q^{\can}$;
    \item (Base Change) Consider pullback square $(\ast)$ as in Definition \ref{base change}, then if $q$ is (weakly) \catC-ambidextrous, $\tilde{q}$ is (weakly) \catC-ambidextrous, and the associated base-change square $\square$ is (weakly) ambidextrous.
\end{enumerate}
\end{prop}

We end this section with the definition of higher semiadditivity.

\begin{defn}
Let $m\geq -2$ be an integer, and consider \inftycat\ \catC. \catC\ is called $m$-semiadditive, if it admits all $m$-finite (co)limits and every $m$-finite map of spaces is \catC-ambidextrous. It is called $\infty$-semiadditive if it is $m$-semiadditive for all $m\geq -2$.
\end{defn}

\begin{rem}
Note that in the above definition we focus on $m$-finite maps, which is a restriction of $m$-truncated maps as we saw in the definition of canonical norms.
\end{rem}

\subsubsection{Integration}
Having constructed a family of canonically normed functors, we immediately obtain the integration for canonically iso-normed functors. In this section, we apply the previously developed integration theory to these functors.

\begin{rem}
It is clear that when \catC\ is $m$-semiadditive and $p:A\to \pt$ is $m$-finite for $m=-2,-1,0$, the integration is just as in the examples given in the background introduction.
\end{rem}

First recall the homogeneity Proposition \ref{homogeneity} and the analogy of Fubini's theorem \ref{Fubini}, which immediately specifies to the following propositions.

\begin{prop}\label{specified homo}
Let \catC\ be an \inftycat, and consider the map of spaces $q:A\to B$. Let $X, Y, Z\in \Fun(B,\mathcal{C})$.
\begin{enumerate}
    \item For all maps $f:q^*X\to q^*Y$ and $g:Y\to Z$,
    \begin{equation*}
        g\circ \left(\int_qf\right) = \int_q \left(q^*g\circ f\right) \in \Map_{\Fun(B,\mathcal{C})}(X, Z).
    \end{equation*}
    \item For all maps $f:X\to Y$ and $g:q^*Y\to q^*Z$,
    \begin{equation*}
        \left(\int_qg\right)\circ f = \int_q\left(g\circ q^*f\right)\in \Map_{\Fun(B,\mathcal{C})}(X,Z).
    \end{equation*}
\end{enumerate}
\end{prop}
\begin{proof}
    Apply Proposition Proposition \ref{homogeneity} to the normed functor $q^{\can}:\Fun(A,\mathcal{C})\normed\Fun(B,\mathcal{C})$.
\end{proof}

\begin{prop}\label{specified fubini}
Let \catC\ be an \inftycat\ and consider the composition of maps of spaces
\begin{equation*}
    A\oto{p}B\oto{q}C,
\end{equation*}
then assume that $p$ and $q$ are both \catC-ambidextrous, for all $X,Y\in \Fun(C,\mathcal{C})$, and $f:p^*q^*X\to p^*q^*Y$, we have
\begin{equation*}
    \int_{qp}f = \int_q\left(\int_pf\right)\in \Map(X,Y).
\end{equation*}
\end{prop}
\begin{proof}
Apply Proposition \ref{Fubini} to the composition of normed functors
\begin{equation*}
    \Fun(A,\mathcal{C})\normed\Fun(B,\mathcal{C})\normed\Fun(C,\mathcal{C}).
\end{equation*}
\end{proof}

Next we look at the base-change square $(\ast)$ as in Definition \ref{base change}.

\begin{prop}\label{pullback}
Let \catC\ be an \inftycat, and for $(\ast)$, assume that $q$ is \catC-ambidextrous, then by Proposition \ref{ambisquare}, $\tilde{q}$ is also \catC-ambidextrous. For all $X,Y\in\Fun(B,\mathcal{C})$ and $f:q^*X\to q^*Y$, we have
\begin{equation*}
    s_B^*\int_qf= \int_{\qtil}s_A^*f\in \Map_{\Fun(B,\mathcal{C})}(s^*_BX,s^*_BY).
\end{equation*}
In particular, for all $X\in\Fun(B,C)$,
\begin{equation*}
    s_B^*\card{q}_X=\card{\tilde{q}}_{s^*_BX}\in\Map_{\Fun(B,\mathcal{C})}(s^*_BX,s^*_BX).
\end{equation*}
\end{prop}
\begin{proof}
Let $(\square)$ denote the associated base-change square, Proposition \ref{ambisquare} tells us that $(\square)$ is ambidextrous and Lemma \ref{inductive lemma} tells us that it satisfies the $\bc_!$ condition. Therefore Proposition \ref{ambiBC} gives us the result.

\end{proof}

\begin{rem}
Apply the previous proposition to the fiber square
\begin{equation*}
    \begin{tikzcd}
        q^{-1}(b)\arrow{r}\arrow{d}&A\arrow{d}{q}\\
        \pt\arrow{r}{b}&B,
    \end{tikzcd}
\end{equation*}
for $X\in \Fun(B,\mathcal{C})$ we have
\begin{equation*}
    \card{q}_X(b) = \card{q^{-1}(b)}_{X(b)}\in \Map_{\mathcal{C}}(X(b),X(b)). 
\end{equation*}
In other words, $\card{q}$ is just a combination of $B$-family of the cardinalities of all fibers. In fact, \cite[Proposition~4.3.5]{AmbiKn} tells us that the property of a map of spaces $q:A\to B$ being ambidextrous is also determined by the ambidexterity of the fibers of $q$. 
\end{rem}
As a consequence, we have the following corollary, which is a form of ``distributivity'' of integration over pullback squares.
\begin{cor}\label{pullbackint}
Let \catC\ be an \inftycat, $q_1:A_1\to B$ and $q_2:A_2\to B$ be two \catC-ambidextrous maps of spaces. Consider the pullback square
\begin{equation*}
    \begin{tikzcd}
        A_2\times_BA_1\arrow{r}{\pi_1}\arrow[swap]{d}{\pi_2}\arrow[dr, "q_2\times_B q_1" description]&A_1\arrow{d}{q_1}\\
        A_2\arrow[swap]{r}{q_2}&B
    \end{tikzcd}
\end{equation*}
The map $q_1\times_Bq_2$ is \catC-ambidextrous and for all $X,Y,Z\in \Fun(B,\mathcal{C})$ and maps
\begin{equation*}
    f_1:q_1^*X\to q_1^*Y,\quad
    f_2:q_2^*Y\to q_2^*Z,
\end{equation*}
we have
\begin{equation*}
    \int_{q_2\times_Bq_1}(\pi_2^*f_2\circ \pi_1^*f_1) = \int_{q_2}f_2\circ \int_{q_1}f_1 \in \Map_{\Fun(B,\mathcal{C})}(X,Z).
\end{equation*}
In particular,
\begin{equation*}
    \card{q_2\times_Bq_1}_X=\card{q_2}_X\circ \card{q_1}_X \in \Map_{\Fun(B,\mathcal{C})}(X,X).
\end{equation*}
\end{cor}
\begin{proof}
The second equation follows by applying the first to $f_2=q_2^*\Id_X$ and $f_1= q_1^*\Id_X$.

The map $\pi_2$ is \catC-ambidextrous by Proposition Proposition \ref{ambisquare}(2), so by Proposition Proposition \ref{ambisquare}(1) $q_2\times_B q_1= q_2\pi_2$ is \catC-ambidextrous. Then we prove the first equation with the following long equation
\begin{equation*}
\begin{aligned}
\int_{q_2\times_B q_1}(\pi_2^*f_2\circ \pi_1^*f_1) &= \int_{q_2\pi_2}(\pi_2^*f_2\circ \pi_1^*f_1)\\
&= \int_{q_2}\int_{\pi_2}(\pi_2^*f_2\circ \pi_1^*f_1)&\text{ by Proposition \ref{specified fubini}}\\
&= \int_{q_2}\left(f_2\circ \int_{\pi_2}\pi_1^*f_1\right)&\text{ by Proposition \ref{specified homo}(1)}\\
&= \int_{q_2}\left(f_2\circ q_2^*\int_{q_1}f_1\right) &\text{ by Proposition \ref{pullback}}\\
&= \int_{q_2}f_2\circ\int_{q_1}f_1. &\text{ by Proposition \ref{specified homo}(2)}
\end{aligned}
\end{equation*}
\end{proof}
Another important consequence of Corollary \ref{pullbackint} is the additivity of integral.

\begin{prop}
Let \catC\ be a $0$-semiadditive \inftycat, and let $q_i:A_i\to B$ for $i=1,...,k$ be \catC-ambidextrous maps. Then
\begin{equation*}
    (q_1,...,q_k):A_1\sqcup...\sqcup A_k\to B
\end{equation*}
is \catC-ambidextrous, and for all $X,Y\in\Fun(B,\mathcal{C})$, $f_i:q_i^*X\to q_i^*Y$ for $i = 1,...,k$, we have
\begin{equation*}
    \int_{q_1,...,q_k}(f_1,...,f_k)=\sum_{i=1}^k\left(\int_{q_i}f_i\right)\in \Map_{\Fun(B,\mathcal{C})}(X,Y).
\end{equation*}
\end{prop}
\begin{proof}
By induction we can assume $k=2$. Now for $(q_1,q_2)$, write it as the composition
\begin{equation*}
    A_1\sqcup A_2\oto{q_1\sqcup q_2}B_1\sqcup B_2\oto{\nabla} B
\end{equation*}
where $\nabla$ is the fold map. By \cite[Proposition~4.3.5]{AmbiKn}, $q_1\sqcup q_2$ is \catC-ambidextrous. Consider the pullback square
\begin{equation*}
    \begin{tikzcd}
        A_1\arrow{r}{j_1}\arrow[swap]{d}{q_1}&A_1\sqcup A_2\arrow{d}{q_1\sqcup q_2}\\
        B\arrow{r}{j_1}&B\sqcup B,
    \end{tikzcd}
\end{equation*}
where $j_1$ denotes the natural inclusion of the first summand. Applying proposition \ref{pullback} to this pullback square we have
\begin{equation*}
    j_1^*\left(\int_{q_1\sqcup q_2}(f_1,f_2)\right)=\int_{q_1}f_1.
\end{equation*}
And similarly, for the second summand, we have
\begin{equation*}
    j_2^*\left(\int_{q_1\sqcup q_2}(f_1,f_2)\right)=\int_{q_2}f_2.
\end{equation*}
Combining them we get
\begin{equation*}
    \int_{q_1\sqcup q_2}(f_1,f_2)=\left(\int_{q_1}f_1, \int_{q_2}f_2\right).
\end{equation*}
\catC\ is $0$-semiadditive and $\nabla$ is $0$-finite, so $\nabla$ is \catC-ambidextrous, and by Proposition \ref{ambisquare}(1), the composition $(q_1,q_2)$ is also \catC-ambidextrous. Using Proposition \ref{specified fubini} and the definition of integral over $0$-finite map, we have
\begin{equation*}
    \int_{(q_1,q_2)}(f_1,f_2)\cong \int_{\nabla}\int_{q_1\sqcup q_2}(f_1,f_2)
    =\int_{\nabla}\left(\int_{q_1}f_1,\int_{q_2}f_2\right)
    = \int_{q_1}f_1+\int_{q_2}f_2.
\end{equation*}
\end{proof}
Now we move to the interaction of integration and maps of \inftycats. Again we start with a specific commutative diagram.

\begin{defn}
Let $F:\mathcal{C}\to\mathcal{D}$ be a functor of \inftycats, and $q:A\to B$ a map of spaces. The $(F,q)$-square is the following commutative diagram:
\begin{equation*}
\begin{tikzcd}
    \Fun(B,\mathcal{C})\arrow{r}{F_*}\arrow[swap]{d}{q^*}&\Fun(B,\mathcal{D})\arrow{d}{q^*}\\
    \Fun(A,\mathcal{C})\arrow{r}{F_*}&\Fun(A,\mathcal{D}).
\end{tikzcd}
\end{equation*}
Here $F_*$ means post-composition with $F$. If $q$ is weakly \catC\ and \catD ambidextrous, then the square is canonically normed.
\end{defn}

In light of \cite[Proposition~3.2.2]{TeleAmbi} and \cite[Theorem~3.2.3]{TeleAmbi}, the key property of $F$ to be compatible with the integration is the preserving of $q$-colimits.

\begin{prop}[{\cite[Proposition 3.2.2]{TeleAmbi}}]\label{F-q BC}
Let $F:\mathcal{C}\to\mathcal{D}$ be a functor of \inftycats, and $q:A\to B$ a map of spaces. If \catC\ and \catD\ admit all $q$-colimits (resp. $q$-limits), and $F$ preserves all $q$-colimits (resp. $q$-limits), then the $(F,q)$-square satisfies the $\bc_!$ (resp. $\bc_*$) condition.
\end{prop}

\begin{thm}[{\cite[Theorem 3.2.3]{TeleAmbi}}]\label{F-q Ambi}
Let $F:\mathcal{C}\to\mathcal{D}$ be a functor of \inftycats, and $q:A\to B$ a map of spaces. Assume $F$ preserves all $(m-1)$-finite colimits and $q$ is $m$-finite, if $q$ is (weakly) \catC-ambidextrous and (weakly) \catD-ambidextrous, then the $(F,q)$-square is (weakly) \catC-ambidextrous.
\end{thm}

\begin{rem}
It may seem asymmetric in the last theorem, that the condition is on the preserving of colimits only, without mentioning the limits. However, this will be partly resolved by the next corollary.
\end{rem}

\begin{cor}[{\cite[Corollary 3.2.4]{TeleAmbi}}]
Let $F:\mathcal{C}\to\mathcal{D}$ be a functor of $m$-semiadditive \inftycats, then $F$ preserves $m$-finite colimits if and only if it preserves $m$-finite limits.
\end{cor}

Now we are able to define the ``compatible'' functors between $m$-semiadditive \inftycats.

\begin{defn}
Let \catC\ and \catD\ be $m$-semiadditive functors. A functor $F:\mathcal{C}\to\mathcal{D}$ is called $m$-semiadditive, if it preserves all $m$-finite colimits, or equivalently if it preserves all $m$-finite limits.
\end{defn}

\begin{note}
We shall directly say $F:\mathcal{C}\to\mathcal{D}$ is $m$-semiadditive, which implies that \catC\ and \catD\ are $m$-semiadditive.
\end{note}
The precise description of the property of compatibility is given by the following proposition.

\begin{prop}
Let $F:\mathcal{C}\to\mathcal{D}$ be an $m$-semiadditive functor, and $q:A\to B$ a map of spaces. For all $X,Y\in \Fun(B,\mathcal{C})$ and $f:q^*X\to q^*Y$,
\begin{equation*}
    F\left(\int_qf\right)=\int_qF(f)\in \Map_{\Fun(B,\mathcal{D})}(FX,FY).
\end{equation*}
In particular,
\begin{equation*}
    F(\card{q}_X) = \card{q}_{F(X)}\in\Map_{\Fun(B,\mathcal{D})}(FX,FX).
\end{equation*}
\end{prop}
\begin{proof}
The $(F,q)$-square is ambidextrous by Theorem \ref{F-q Ambi}, and satisfies the $\bc$-conditions by Proposition \ref{F-q BC}, so we can apply Proposition \ref{ambiBC} to the $(F,q)$-square and get the wanted equations.
\end{proof}

\subsubsection{Monoidal Structure}
In this section, we study the interaction of the property of higher semiadditivity with the monoidal structures.

Consider \inftycat\ \catC\ with a monoidal structure $(\mathcal{C},\otimes,\one)$. For every space $A$, the \inftycat\ $\Fun(A,\mathcal{C})$ has a canonical point-wise monoidal structure. For a map of spaces $q:A\to B$, the functor $q^*:\Fun(B,\mathcal{C})\to\Fun(A,\mathcal{C})$ is canonically monoidal.

\begin{rem}
If the monoidal structure is symmetric, then the associated monoidal structures are also canonically symmetric.
\end{rem}

\begin{prop}[{\cite[Proposition 3.3.1]{TeleAmbi}}]
Let $(\mathcal{C},\otimes,\one)$ be a monoidal \inftycat, let $q:A\to B$ be a \catC-ambidextrous map of spaces, such that $\otimes$ distributes over $q$-limits, then the normed functor
\begin{equation*}
    q^{\can}:\Fun(A,\mathcal{C})\to\Fun(B,\mathcal{C})
\end{equation*}
is $\otimes$-normed (as defined in Definition \ref{defn tensor normed}) in a canonical way.
\end{prop}

From the last proposition, we define the ``compatibly'' monoidal $m$-semiadditive \inftycat.

\begin{defn}
Let $(\mathcal{C},\otimes,\one)$ be a monoidal \inftycat\ which is $m$-semiadditive, it is called an $m$-semiadditive (symmetric) monoidal \inftycat\ if $\otimes$ distributes over $m$-finite colimits.
\end{defn}

\begin{prop}
Let $(\mathcal{C},\otimes,\one)$ be an $m$-semiadditive monoidal \inftycat, and $A$ an $m$-finite space.
\begin{enumerate}
    \item For every $X\in\mathcal{C}$, we have
    \begin{equation*}
        \card{A}_X = \Id_X\otimes\card{A}_{\one}\in \Map_{\mathcal{C}}(X,X).
    \end{equation*}
    \item $\card{A}$ is an isomorphism if and only if $\card{A}_{\one}$ is an isomorphism.
\end{enumerate}
\end{prop}
\begin{proof}
For (1), for a fixed object $X\in \mathcal{C}$, consider the functor
\begin{equation*}
    F_X:\mathcal{C}\to\mathcal{C}:Y\mapsto X\otimes Y.
\end{equation*}
Note that $F_X$ preserves $m$-finite colimits by the $m$-semiadditivity of \catC, so by applying Proposition \ref{pullback} to $F_X$ we have
\begin{equation*}
    \Id_X\otimes \card{A}_{\one}= F_X(\card{A}_{\one})= \card{A}_{F_X(\one)}=\card{A}_X.
\end{equation*}
(2) follows immediately from (1).
\end{proof}

\begin{note}
If \catC\ is an $m$-semiadditive monoidal \inftycat, from the previous proposition, for $m$-finite space A, we can identify $\card{A}$ with an element of $\pi_0(\one)=\pi_0(\Map_{\mathcal{C}}(\one,\one))$.
\end{note}

Note that tensor product can be viewed as an action of \catC\ on \catC, and it is not surprising that we can generalize the above proposition to the following.

\begin{prop}\label{ModuleCard}
Let $(\mathcal{C},\otimes,\one)$ be an $m$-semiadditive monoidal \inftycat, and $A$ an $m$-finite space. Consider \catC-linear \inftycat\ \catD, which is an \inftycat\ with an action of \catC, and the action is written as the functor $F:\mathcal{C}\to\Fun(\mathcal{D},\mathcal{D})$, and $F$ preserves colimits. Denote the cardinality of $A$ valued in \catC\ by $\card{A}^{\mathcal{C}}$, and in \catD\ by $\card{A}^{\mathcal{D}}$.
\begin{enumerate}
    \item For every $X\in\mathcal{D}$, we have
    \begin{equation*}
        \card{A}_{X}^{\mathcal{D}}=\card{A}_{\one}^{\mathcal{C}}\cdot \Id_X.
    \end{equation*}
    Here $(\cdot)$ denotes the action of $\mathcal{C}$ on $\mathcal{D}$.
    \item $\card{A}^{\mathcal{D}}$ is an isomorphism if $\card{A}^{\mathcal{C}}_{\one}$ is an isomorphism.
\end{enumerate}
\end{prop}
\begin{proof}
    For (1), for a fixed object $X\in\mathcal{D}$, consider the functor
    \begin{equation*}
        F_X:\mathcal{C}\to \mathcal{D}:Z\cdot X
    \end{equation*}
    Note that $F_X$ preserves $m$-finite limits by definition, so applying Proposition \ref{pullback} to $F_X$ we have
    \begin{equation*}
        \card{A}_{\one}^{\mathcal{C}}\cdot \Id_X 
        = F_X(\card{A}_{\one}^{\mathcal{C}})
        = \card{A}_{F_X(\one)}^{\mathcal{D}}
        =\card{A}_X^{\mathcal{D}}.
    \end{equation*}
    
(2) follows immediately from (1).
\end{proof}
We conclude with a brief introduction of the principal fiber sequence.

\begin{defn}
A map of spaces $q:A\to B$ is called principal if it can be extended to a fiber sequence $A\oto{q}B\oto{f}E$.
\end{defn}

Note that all the fibers of a principal sequence are isomorphism.
\begin{prop}\label{principal fiber}
Let \catC\ be an \inftycat, and $q:A\to B$ be a principal \catC-ambidextrous map of \catC-ambidextrous spaces with fiber $F$. For $X\in \mathcal{C}$, if $\card{F}_X$ is invertible, then
\begin{equation*}
    \card{q}_{B^*X}=\card{F}_{B^*X}\in \Map_{\Fun(B,\mathcal{C})}(B^*X,B^*X),
\end{equation*}
and
\begin{equation*}
    \card{A}_X=\card{F}_X\card{B}_X.
\end{equation*}
\end{prop}
\begin{proof}
    Since $q$ is principal, the base change of $q$ along itself $A\times_B A\oto{\tilde{q}}A$ is principal with a section, so $\tilde{q}$ is isomorphic to projection $A\times F\oto{\pi_A}A$. From the pullback square
    \begin{equation*}
        \begin{tikzcd}
            A\times F\arrow{r}\arrow[swap]{d}{\pi_A}
            &B\times F\arrow{d}{\pi_B}\\
            A\arrow{r}&B
        \end{tikzcd}
    \end{equation*}
    we know $q\times_Bq\cong q\times_B \pi_B$. From Corollary \ref{pullbackint},
    \begin{equation*}
        \card{q}^2=\card{q\times_B q}
        =\card{q\times_B \pi_B}
        =\card{q}\card{\pi_B}.
    \end{equation*}
    For $X\in\mathcal{C}$, if $\card{F}_X$ is invertible, then $\card{q}_{B^*X}$ is invertible, so by Proposition \ref{pullback} we have
\begin{equation*}
    \card{q}_{B^*X}=\card{\pi_B}_{B^*X}=B^*(\card{F}_X),
\end{equation*}
so integrate along $B$, with Proposition \ref{specified homo} we have
\begin{equation*}
    \card{A}_X=\int_B\card{q}_{B^*X}=\int_BB^*(\card{F}_X)=\card{F}_X\card{B}_X.
\end{equation*}
\end{proof}

We could generalize the previous proposition to the following relative version, which will be used in the last chapter.

\begin{prop}\label{relative principal}
Consider a map of spaces $A\oto{q} B$, suppose it is induced by a $G$-equivariant principal fibration $F\to\tilde{A}\oto{\tilde{q}} \tilde{B}$ by taking the $G$-quotient, then if $\card{F/G\to BG}$ is invertible, we have
\begin{equation*}
    \card{A\to BG} = \card{F/G\to BG}\card{B\to BG}.
\end{equation*}
\end{prop}
\begin{proof}
The proof is analogous to the previous one. As $\tilde{q}$ is principal, $\tilde{A}\times_{\tilde{B}}\tilde{A}\to A$ is a principal map with a section, so it is isomorphic to the projection $F\times \tilde{A}\to \tilde{A}$.

In particular, we have the following pullback squares:
\begin{equation*}
\begin{tikzcd}
    \tilde{A}\times_{\tilde{B}}\tilde{A}\arrow{r}\arrow{d}&\tilde{A}\arrow{d}\\
    \tilde{A}\arrow{r}&\tilde{B},
\end{tikzcd}
\qquad\qquad
\begin{tikzcd}
    \tilde{A}\times F\arrow{r}\arrow{d}&\tilde{B}\times F\arrow{d}{\pi_{\tilde{B}}}\\
    \tilde{A}\arrow{r}&\tilde{B},
\end{tikzcd}
\end{equation*}
where $\tilde{A}\times_{\tilde{B}}\tilde{A}\to \tilde{B}$ isomorphic to $\tilde{A}\times F\to \tilde{B}$.

Taking the $G$-quotient we still have two pullback squares
\begin{equation*}
\begin{tikzcd}
    A\times_BA\arrow{r}\arrow{d}&A\arrow{d}\\
    A\arrow{r}&B,
\end{tikzcd}
\qquad\qquad
\begin{tikzcd}
    A\times_{BG}F/G\arrow{r}\arrow{d}&B\times_{BG}F/G\arrow{d}{\pi_B}\\
    A\arrow{r}&B,
\end{tikzcd}
\end{equation*}
and in particular
\begin{equation*}
    q\times_Bq\cong q\times_B\pi_B.
\end{equation*}
Then from Corollary \ref{pullbackint},
\begin{equation*}
    \card{q}^2 = \card{q}\circ \card{\pi_B}.
\end{equation*}
For $X\in \Fun(BG, \mathcal{C})$, if $\card{F/G\to BG}_X$ is invertible, then $\card{q}_{B^*X}$ is invertible, so by Proposition \ref{pullback} we have
\begin{equation*}
    \card{q}_{B^*X}=\card{\pi_B}_{B^*X}.
\end{equation*}
Integrating along $B\to BG$, with Proposition \ref{specified homo} we have
\begin{equation*}
    \card{A\to BG}_X = \int_{B\to BG}\card{q}_{B^*X} = \int_{B\to BG}\card{\pi_B}_{B^*X}=\card{F/G\to BG}_X\circ\card{B\to BG}_X.
\end{equation*}
\end{proof}

\newpage

\section{Height}
In this section, we give a brief introduction to the notion of semiadditive height. We will focus on the definition of semiadditive height in the first section, and reference several propositions and theorems to justify the definition in the second.
\subsection{Semiadditive Height}
The definition of height requires a fixed prime $p$, so we first define the category of $p$-typical $m$-semiadditive \inftycats.
\begin{defn}
Let $p$ be a prime and $0\leq m\leq \infty$.
\begin{enumerate}
    \item An \inftycat\ \catC\ is $p$-typically $m$-semiadditive if all $m$-finite $p$-spaces are \catC-ambidextrous.
    \item A functor $F:\mathcal{C}\to \mathcal{D}$ between $p$-typically $m$-semiadditive if it preserves all $m$-finite $p$-space colimits.
    \item A (symmetric) monoidal \inftycat\ \catC\ is $p$-typically $m$-semiadditive (symmmetric) monoidal if it is $p$-typically $m$-semiadditive and is compatible with $m$-finite $p$-space colimits.
\end{enumerate}
\end{defn}
\begin{rem}
Normally we would just assume that \catC\ is $m$-semiadditive, so that we could define the properties about height for all prime $p$.
\end{rem}

\begin{prop}[{\cite[Proposition~3.1.2]{AmbiHeight}}]
 A $0$-semiadditive \inftycat\ \catC\ is $p$-typically $m$-semiadditive if and only if $B^kC_p$ is \catC-ambidextrous for all $k=1,...,m$.
\end{prop}

\begin{note}
We may denote the category of $p$-typically $m$-semiadditive \inftycats\ by $\cat^{\oplus_p-m}$.
\end{note}

As indicated in the previous proposition, in a $p$-typically $m$-semiadditive \inftycat, the cardinality of $\card{B^kC_p}$ plays an important role, so we define a special notation for it.

\begin{note}
Let $\mathcal{C}\in \cat^{\oplus_p-m}$, for every integer $0\leq k\leq m$, we use $p_{(k)}^{\mathcal{C}}$ to denote $\card{B^kC_p}$. We shall omit \catC\ and simply write $p_{(k)}$ whenever \catC\ is clear from the context.
\end{note}

Now we give a example of $\mathcal{O}_n\coloneqq \Mod_{E_n}(\Sp_{K(n)})$, where $K(n)$ is the Morava $K$-theories (See \cite{hoveymorava} or \cite{LurieChrome}).

\begin{example}[{\cite[Proposition 2.2.5]{AmbiHeight}}]
For all $n,k\leq 0$, we have
\begin{equation*}
    p_{(k)}^{\mathcal{O}_n}=p^{\binom{n-1}{k}}.
\end{equation*}
\end{example}

\begin{defn}
Consider an \inftycat\ \catC, let $\alpha:\Id_{\mathcal{C}}\to \Id_{\mathcal{C}}$ be a natural endomorphism. An object $X\in \mathcal{C}$ is
\begin{enumerate}
    \item $\alpha$-divisible if $\alpha_X$ is invertible.
    \item $\alpha$-complete if $\Map(Z, X)\cong \pt$ for all $\alpha$-divisible $Z\in \mathcal{C}$.
\end{enumerate}
\end{defn}

With the above definition and $p_{(n)}$ we are able to define the (semiadditive) height.

\begin{defn}
Let $\mathcal{C}\in \cat^{\oplus_p-m}$, and $0\leq n\leq m\leq \infty$. For every object $X\in \mathcal{C}$ we define the height of $X$ (denoted by $\htt_{\mathcal{C}}(X)$) as:
\begin{enumerate}
    \item $\htt_{\mathcal{C}}(X)\leq n$, if $X$ is $p_{(n)}^{\mathcal{C}}$-invertible.
    \item $\htt_{\mathcal{C}}(X)> n$, if $X$ is $p_{(n)}^{\mathcal{C}}$-complete.
    \item $\htt_{\mathcal{C}}(X)= n$, if $\htt_{\mathcal{C}}(X)\leq n$ and $\htt_{\mathcal{C}}(X)> n-1$.
\end{enumerate}
We say the height of \catC\ (denoted by $\Ht(\mathcal{C})$) is $\leq n$, ($>n$, $=n$), if for all $X\in \mathcal{C}$, $\htt_{\mathcal{C}}(X)\leq n$ ($\htt_{\mathcal{C}}(X)> n$, $\htt_{\mathcal{C}}(X)= n$).
\end{defn}

\begin{rem}
We say $X\in\mathcal{C}$ is of height $\infty$ if and only if $\htt_{\mathcal{C}}(X)>k$ for all $k\geq 0$. When the \inftycat\ \catC\ is clear, we shall write simply write $\htt(X)$.
\end{rem}

The following proposition states the compatibility of height and monoidal structure. 
\begin{prop}[{\cite[Corollary 3.1.16]{AmbiHeight}}]
Let \catC\ be a $p$-typically $m$-semiadditive monoidal \inftycat, for every $0\leq n\leq m$, we have $\Ht(\mathcal{C})\leq n$ if and only if $\htt(\one)\leq n$.
\end{prop}

\begin{example}
Let \catC\ be $0$-semiadditive, then $p_{(0)}$ is just multiply by $p$. An object $X\in \mathcal{C}$ is of height $0$ if and only if $p\cdot\Id_X$ is invertible, and of height $>0$ if and only if $X$ is $p$-complete. Therefore it is easy to see that a non-zero object can have height $>0$ for at most one prime $p$. In particular, if \catC\ is $p$-local, every object has height $0$ for each prime $q\neq p$. In fact, most of the time we are interested in the height of $p$-local \inftycats.
\end{example}

\subsection{Bounded Height}
In this section, we give a few propositions that justify the definition of semiadditive height and then focus on the $p$-local \inftycats. We conclude with a brief introduction of the theory of modes, which gives a universal \inftycat\ that characterizes the property of being stable, $p$-local, $\infty$-semiadditive, and of height $n$.

\begin{prop}[{\cite[Proposition 3.1.9]{AmbiHeight}}]
Let $\mathcal{C}\in \cat^{\oplus_p-m}$, and let $0\leq n_0\leq n_1\leq m$, then for every $X\in \mathcal{C}$,
\begin{enumerate}
    \item If $\htt(X)\leq n_0$, $\htt(X)\leq n_1$.
    \item If $\htt(X)>n_1$, $\htt(X)>n_0$.
\end{enumerate}
\end{prop}
\begin{rem}
The proposition above ensures that there are no examples of $X\in\mathcal{C}$ with contradictory heights, such as $\htt(X)\leq n_0$ but $\htt(X)>n_1$. Also, it is obvious that a nonzero object $X$ could not be of height $\leq n$ and $>n$. However, we may have examples such that $\htt(X)\nleq n$ and $\htt(X)\ngtr n$. For stable \inftycats, this could be partly resolved by the following theorem.
\end{rem}

\begin{thm}[{\cite[Theorem 4.2.7]{AmbiHeight}}]
Let \catC\ be stable $m$-semiadditive \inftycat.
\begin{enumerate}
    \item For $m<\infty$, assume that \catC\ is idempotent complete (see Section 4.4.5 of \cite{htt}), and let $\mathcal{C}_0,...,\mathcal{C}_{m-1},\mathcal{C}_{>m-1}$ be the full subcategories generated by objects of height $0,...,m-1,>m-1$ respectively, then the inclusions determine an equivalence
    \begin{equation*}
        \mathcal{C}\cong \mathcal{C}_0\times...\times \mathcal{C}_{m-1}\times \mathcal{C}_{>m-1}.
    \end{equation*}
    Moreover, $\mathcal{C}_0$ is $p$-typically $\infty$-semiadditive and $\mathcal{C}_1,...,\mathcal{C}_{m-1}$ are $\infty$-semiadditive.
    \item For $m=\infty$, assume \catC\ admits sequential limits and colimits, if $\mathcal{C}_{\infty}=0$, then
    \begin{equation*}
        \mathcal{C}\cong \prod_{n\in\NN}\mathcal{C}_n.
    \end{equation*}
\end{enumerate}
\end{thm}

\subsubsection{The $p$-local case}
In this subsection, we focus on the $p$-local \inftycats, which is a rather strong assumption and gives us further simplifications.

\begin{prop}[{\cite[Proposition 3.2.6]{AmbiHeight}}]
Let \catC\ be a $0$-semiadditive $p$-local \inftycat\ which admits all $1$-finite limits and colimits. \catC\ is $p$-typically $m$-semiadditive if and only if it is $m$-semiadditive.
\end{prop}

The main result for $p$-local \inftycats\ is the following theorem.

\begin{thm}[{\cite[Theorem 3.2.7]{AmbiHeight}}]
Let \catC\ be an $n$-semiadditive $p$-local \inftycat, which admits all $\pi$-finite limits and colimits. If \catC\ is of height $\leq n$, then
\begin{enumerate}
    \item \catC\ is $\infty$-semiadditive.
    \item For every $(n-1)$-connected nilpotent $\pi$-finite space $A$, the map $\card{A}$ is invertible.
    \item For every $n$-connected $\pi$-finite space $A$ and $X\in\mathcal{C}$, the fold map $A\otimes X\oto{\nabla}X$ is invertible.
\end{enumerate}
\end{thm}
\begin{rem}
In formally speaking, (1) says that being $n$-semiadditive trivializes the higher semiadditive structure above $n$. (2) is an application of Proposition \ref{principal fiber} on the Postnikov (see \cite[Section~4.3]{HatcherTop}) tower of $A$. (3) gives a useful categorical result without referring to the higher semiadditive structure directly.
\end{rem}

\subsubsection{Mode Theory}
In \cite[Section~4.8.2]{HA}, an idempotent object $X$ of a symmetric monoidal \inftycat\ \catC\ is defined by a morphism $\one\oto{u}X$, such that
\begin{equation*}
    X\otimes \one\oto{1\otimes u}X\otimes X
\end{equation*}
is an isomorphism. Moreover, the idempotent object $X$ has a unique commutative algebra structure for which $u$ is the unit, so we also call $X$ an idempotent algebra, denoted by $R$ to emphasize the commutative ring structure. It is indicated in \cite[Proposition~4.2.8.10]{HA} that it is a property to have the structure of an $R$-module. We say that $R$ classifies or characterizes the property of being an $R$-module.

\begin{example}
When $\mathcal{C}=\Ab$, the idempotent algebras are called solid rings as in \cite[Definition~2.1]{CoreRing}. An example of this is $\rational$, and it classifies the property of being rational.
\end{example}

The \inftycat\ of presentable \inftycats\ and colimit-preserving functors, denoted by $\Pr^L$, has a symmetric monoidal structure by \cite[Proposition~4.8.1.15]{HA}. We call the idempotent algebra in $\Pr^L$ modes (also considered in \cite{GepnerUniversality} with smashing localizations).

\begin{example}
The category $\Sp$ is a mode, which classifies the property of being stable.
\end{example}

\begin{thm}[{\cite[Corollary 5.20]{ambiSpan}}]
The forgetful functor $\Mod_{\CMon_m(\Spc)}(\Pr_L)\to \Pr_L$ is fully faithful and its essential image are the $m$-semiadditive presentable \inftycats.
\end{thm}
Here $\Spc$ means the \inftycat\ of spaces, so $\CMon_m(\Spc)$ means the idempotent algebra of $m$-commutative monoids in spaces. The theorem tells  us that $\CMon_m(\Spc)$ classifies the property of being $m$-semiadditive, and we have a similar theorem for height.

\begin{thm}[{\cite[Theorem 5.3.6]{AmbiHeight}}]
For every $n\geq 0$, there exists a mode $\tsadi_n$, which characterizes the property of being stable, $p$-local, $\infty$-semiadditive and of height $n$.
\end{thm}

\begin{rem}\label{WLOGsymmetric}
In particular, $\tsadi_n$ is symmetric monoidal, so in light of Proposition \ref{ModuleCard}, for the calculation of cardinality with value in \catC, which is stable, $p$-local, $\infty$-semiadditive and of height $n$, we may assume \catC\ is symmetric monoidal.
\end{rem}
There are more nontrivial phenomena and theorems in the theory of semiadditive height. For example, \cite[Theorem~3.3.2]{AmbiHeight} states gives a description of \emph{Semiadditive redshift}, in analogy to the \emph{chromatic redshift} discovered in \cite{ausoni2008chromatic}. We stop here and reference to $\cite{AmbiHeight}$ for further study of this field.
\newpage

\section{\mobius Inversion and Burnside Ring}
In this section, we give a brief note of \mobius function and inversion formula, Burnside ring, and the idempotents of a Burnside ring. The outline of the section of \mobius function is based on \cite{Rota}, the basic discussion of Burnside ring follows from \cite[Chapter~5]{RepTheory}, and the expression for the primitive idempotents is mainly based on \cite{BurnsideIdem}.

\subsection{\mobius Inversion for poset}
In this section, we give a brief introduction of the \mobius function and the \mobius inversion formula of locally finite posets. The classical form of the \mobius inversion formula originates from number theory, and we will show that it is a specific case of the general theory applied to the poset of $\integer$, where the partial order is given by divisibility.

First, we define the incidence algebra which allows us to manipulate the calculation in simple notations.

\begin{defn}
For a locally finite poset $(P,<)$, and a commutative ring $R$, we define the incidence algebra $I$ of $P$ relative to $R$ to be the set of functions $f:P\times P\to R$, such that $x\nleq y \Rightarrow f(x,y)=0$. We have pointwise addition and scalar multiplication
\begin{equation*}
    (f+g)(x,y)=f(x,y)+g(x,y),\quad (r\cdot f)(x,y)=r\cdot(f(x,y)),
\end{equation*}
and the convolution product
\begin{equation*}
    (f\ast g)(x,y) = \sum_{x\leq z\leq y}f(x,z)\cdot g(z,y).
\end{equation*}
\end{defn}
\begin{rem}\label{incidence and matrix}
It is sometimes more intuitive to think of $f\in I$ as an $R$-linear endomorphism $T_f$ of $R^P\coloneqq\Map(P,R)$
\begin{equation*}
    T_f(e_x)=\sum_{x\leq y}f(x,y)e_y,
\end{equation*}so the definitions of addition, scalar multiplication, and convolution product are consistent with the addition, scalar multiplication, and the composition of $R$-linear maps.
\end{rem}
The identity of the convolution product is given by $\delta$:
\begin{equation*}
\delta(x,y)=\left\{
    \begin{aligned}
    &1,\text{ if } x=y\\
    &0,\text{ otherwise.}
    \end{aligned}
    \right.
\end{equation*}
And we also have the function $\zeta$:
\begin{equation*}
\zeta(x,y)=\left\{
\begin{aligned}
    &1,\text{ if }x\leq y\\
    &0,\text{ otherwise.}
    \end{aligned}
    \right.
\end{equation*}
The \mobius function is the  multiplicative inverse of the zeta function. We show in the following proposition, using the fact that the poset $P$ is locally finite, that $\zeta$ is multiplicatively invertible.

\begin{prop}
If $P$ is locally finite, namely if $\{z|x\leq z\leq y\}$ is finite for all $x, y\in P$, then there exists a $\mu\in I$ such that $\mu\ast\zeta=\zeta\ast\mu=\delta$.
\end{prop}
\begin{proof}
We construct the value of $\mu(x,y)$ inductively on the number of elements in $\{z|x\leq z\leq y\}$. For the base case we set $\mu(x,x)=0$ for all $x\in P$, and if $x\neq y$, assume that $\mu(x,z)$ is defined for all $\{z|x\leq z< y\}$, then we define
\begin{equation*}
    \mu(x,y)=-\sum_{x\leq z<y}\mu(x,z).
\end{equation*}
Note that 
\begin{equation*}
\begin{aligned}
    (\mu\ast\zeta)(x,y)&=\sum_{x\leq z\leq y}\mu(x,z)\\
    (\zeta\ast\mu)(x,y)&=\sum_{x\leq z\leq y}\mu(z,y),
\end{aligned}
\end{equation*}
so by definition $\sum_{x\leq z\leq y}\mu(x,z)=0$ unless $x=y$.
Also, we have
\begin{equation*}
\begin{aligned}
    \sum_{x\leq z\leq y}\mu(z,y)
    &=\mu(y,y)-\sum_{x\leq z\leq y}\sum_{z\leq w<y}\mu(z,w)\\
    &=1-\sum_{x\leq w<y}\sum_{x\leq z\leq w}\mu(z,w)\\
    &=1-\sum_{x<w<y}\sum_{x\leq z\leq w}\mu(z,w) - \mu(x,x)\\
    &=\sum_{x<w<y}\sum_{x\leq z\leq w}\mu(z,w).
\end{aligned}
\end{equation*}
There for if we assume for all $w$ strictly between $x$ and $y$, $\sum_{x\leq z\leq y}\mu(z,w)=0$, then we have $\sum_{x\leq z\leq y}\mu(x,y)=0$. The base case is when $\{z|x<z<y\}=\varnothing$, then
\begin{equation*}
\mu(x,y)+\mu(y,y)=-\mu(x,x)+\mu(y,y)=0.
\end{equation*}
Therefore by induction, we have for all $x\neq y$
\begin{equation*}
    \sum_{x\leq z\leq y}\mu(z,y)=0,
\end{equation*}
so $\mu$ is indeed an inverse of $\zeta$. 
\end{proof}

\begin{rem}
There is a more ``geometric'' construction of $\mu(x,y)$. For $x\neq y$, consider the full subcategory of $P$ generated by elements in $(x, y)$, then take the geometric realization of the nerve. Then $\mu(x,y)$ is just the reduced Euler characteristic of this space. Here by ``reduced'' we mean the usual Euler characteristic minus $1$, corresponding to the augmented chain complex with $\integer$ in dimension $-1$.
\end{rem}

\begin{example}
Let $P(X)$ be the poset of subsets of a finite set $X$, then for $I\leq J$ we have
\begin{equation*}
    \mu(I,J)= (-1)^{\card{J-I}}.
\end{equation*}
\end{example}

\begin{example}\label{number theoretic}
Another example we see regularly is number theoretic, namely on the set of positive natural numbers, with the order given by divisibility. For $x|y$ we have
\begin{equation*}
    \mu(x,y)=\mu(1,y/x)=
    \begin{cases}
    (-1)^r, &y/x \text{ is the product of r distinct primes}\\
    0, &\text{otherwise}.
    \end{cases}
\end{equation*}
\end{example}

Recall from Remark \ref{incidence and matrix} that we could view $\zeta$ and $\mu$ as $R$-linear endomorphisms of $R^P$, them being inverse to each other immediately gives us the following corollary of \mobius inversion formula.

\begin{cor}
Let $f,g\in R^P$, then
\begin{equation*}
    f(y)=\sum_{x\leq y}g(x)\Longleftrightarrow g(y)= \sum_{x\leq y}\mu(x,y)f(x);
\end{equation*}
dually we have
\begin{equation*}
    f(x)=\sum_{y\geq x}g(x) \Longleftrightarrow g(x) = \sum_{y\geq x}\mu(x,y)f(y).
\end{equation*}
\end{cor}

\begin{rem}
Apply the above formula to Example \ref{number theoretic}, we have the familiar \mobius inversion formula for number theory.
\end{rem}
\subsection{Burnside Ring}
In this section, we give a brief introduction to the basic properties of the Burnside ring of a finite group $G$.
\begin{note}
For finite group $G$, let $\Omega^+(G)$ denote the commutative unital semi-ring of the $G$-isomorphism classes of finite $G$-sets. In particular, the addition is the direct sums, the multiplication is the Cartesian product, and the multiplicative identity is the trivial $G$-set.
\end{note}

\begin{defn}
The $Grothendieck\ ring$ $\bar{A}$ of a semi-ring $A$ is the formal construction that gives each element $A$ an additive inverse. Explicitly, it is defined to be the set of equivalence classes of Cartesian product $A\times A$, under the relation
\begin{equation*}
    (a_+, a_-)\sim(b_+,b_-)\Longleftrightarrow \exists c\in A,\ a_++b_-+c=b_++a_-+c.
\end{equation*}
The addition is given by
\begin{equation*}
    [a_+,a_-]+[b_+,b_-]=[a_++b_+,a_-+b_-],
\end{equation*}
and the multiplication is given by
\begin{equation*}
    [a_+,a_-]\cdot [b_+,b_-]=[a_+b_++a_-b_-,a_+b_-+a_-b_+].
\end{equation*}
There is a canonical morphism of semi-ring from $A$ to $\bar{A}$, given by
    $a\mapsto [a,0]$.
\end{defn}

\begin{defn}
For a finite group $G$, the $Burnside\ ring$ $\Omega(G)$ of $G$ is the Grothendieck ring of $\Omega^+(G)$.
\end{defn}
In many cases, it is useful to localize the Burnside ring, as in the following notation.
\begin{note}
For a finite $G$-set $X$, denoted its image in $\Omega(G)$ by $[X]$. For prime $p$, let $\integer_{(p)}$ be the $p$-localization of $\integer$, we denote by $\Omega_p(G)$ the $p$-localization, namely the tensor product $\Omega(G)\otimes\integer_{(p)}$, and similarly by $\Omega_{\rational}(G)$ the tensor product $\Omega(G)\otimes \rational$.

As an additive group, $\Omega(G)$ is free with basis $\{[G/H]|(H)\in C(G)\}$, where $C(G)$ denotes the set of conjugate classes of subgroups of $G$. We set
\begin{equation*}
    \tilde{\Omega}(G)\coloneqq \prod_{(H)\in C(G)}\integer,
\end{equation*}
and similarly
\begin{equation*}
    \begin{aligned}
       \tilde{\Omega}_{p}(G)&\coloneqq \prod_{(H)\in C(G)}\integer_p,\\
       \tilde{\Omega}_{\rational}(G)&\coloneqq\prod_{(H)\in C(G)}\rational.
    \end{aligned}
\end{equation*}
\end{note}

Now we give an injection of the Burnside ring into $\tilde{\Omega}(G)$, which is much simpler as a direct product of copies of $\integer$.

Fix a subgroup $H$ of $G$, we have a map:
\begin{equation*}
    \phi_H:\Omega(G)\to \integer:\ 
    [X]\mapsto \card{X^H},
\end{equation*}
where $X^H$ denotes the set of $H$-fixed elements of X. It is easy to verify that $\phi_H$ is a ring homomorphism. Moreover, we have the following ring homomorphism:
\begin{equation*}
    \phi\coloneqq \prod_{(H)\in C(G)}\phi_H:
    \Omega(G)\to\tilde{\Omega}(G).
\end{equation*}
For prime $p$ and $\rational$, we can define $\phi_{H,p}$ and $\phi_{H,\rational}$ as the extension of scalars of $\phi_H$ to $\integer_{(p)}$ and $\rational$ respectively, namely the tensor product of $\phi$ with $\integer_{(p)}\oto{\Id}\integer_{(p)}$ and $\rational\oto{\Id}\rational$. Similarly we define $\phi_p$ and $\phi_{\rational}$.

\begin{prop}\label{subring}
The ring homomorphisms $\phi$, $\phi_{p}$ and $\phi_{\rational}$ are injective.
\end{prop}
\begin{proof}
Since for each ideal $I$ in $\integer_{(p)}$ or $\integer_{\rational}$, we have
\begin{equation*}
    I\cap \integer = (0)\Longleftrightarrow I=(0),
\end{equation*}
therefore the extension of scalars to $\integer_{(p)}$ and $\rational$ preserve injectivity, so it suffices to prove that $\phi$ is injective.

Suppose $x\neq 0$ is in the kernel of $\phi$, we rewrite $x$ using the additive basis:
\begin{equation*}
    x = \sum a_H[G/H].
\end{equation*}
Note that there is a partial order on $\{[G/H]|H\leq G\}$, namely $[G/H_1]\leq [G/H_2]$ if and only if $H_1$ is conjugate to a subgroup of $H_2$. Let $[G/H]$ be maximal among the elements in $x$ with $a_H\neq 0$. Since $(G/K)^H\neq \varnothing\Rightarrow [G/H]\leq [G/K]$, we have
\begin{equation*}
    0 = \phi_Hx= a_H(G/H)^H=a_H\card{NH/H}\neq 0,
\end{equation*}
where $NH$ is the normalizer of $H$, this is a contradiction.
\end{proof}
\begin{rem}
Since $\phi,\phi_{(p)}$ and $\phi_{\rational}$ are injective, we can view $\Omega(G), \Omega_p(G)$ and $\Omega_{\rational}(G)$ as subring of $\tilde{\Omega}(G)$, $\tilde{\Omega}_p(G)$ and $\tilde{\Omega}_{\rational}(G)$ respectively, which helps with the study of idempotent elements in the next section.
\end{rem}

\subsection{Idempotents}
In this section, we describe the primitive idempotents in the Burnside ring. We first study the congruence relation of elements in the Burnside ring, use the relation to characterize the primitive idempotents, and then give an expression for the primitive idempotents.

First, let us look at some general facts about idempotents in commutative ring.
\begin{defn}
For commutative ring $R$, an idempotent $a\in R$ (short for idempotent element) is an element of $R$ such that $a\cdot a=a$. Two idempotents $a$ and $b$ are orthogonal if $ab=0$. A nonzero idempotent $c$ is called primitive if it cannot be written as $c=a+b$, where $a$ and $b$ are nonzero orthogonal idempotents.
\end{defn}

\begin{lem}
Any two primitive idempotents of $R$ are orthogonal.
\end{lem}
\begin{proof}
Consider primitive idempotents $a,b\in R$. If $ab\neq0$, then either $ab\neq a$ or $ab\neq b$. Without loss of generality, assume $ab\neq a$, then $ab$ and $a-ab$ are both nonzero idempotent, and they are orthogonal. So $a=ab+(a-ab)$ contradicts the fact that $a$ is primitive.
\end{proof}

For commutative ring $R$, we call the set of all primitive idempotents complete, if any idempotent can be decomposed into a finite sum of primitive idempotents.

\begin{lem}
If $R$ has a finite complete set of primitive idempotents $\{a_1,...,a_k\}$, then $1=a_1+...+a_k$.
\end{lem}
\begin{proof}
The element $1-(a_1+...+a_k)$ is idempotent, and it is orthogonal to $a_1,...,a_k$. If it is nonzero, then it equals a finite sum of primitive idempotents, which is not possibly orthogonal to all the primitive idempotents.
\end{proof}

\begin{lem}
If a commutative ring $R$ has a finite complete set of primitive idempotents, so does any subring $R'$ of $R$.
\end{lem}
\begin{proof}
An idempotent of $R'$ is also an idempotent in $R$, and $R$ has finitely many idempotents.
\end{proof}
\begin{prop}
For finite group $G$, the burnside ring $\Omega(G)$ and $p$-local burnside ring $\Omega_p(G)$ with prime $p$ have a finite complete set of primitive idempotents.
\end{prop}
\begin{proof}
It is clear that for finite set $I$, $\prod_I\integer$ or $\prod_I\integer_{(p)}$ has a finite complete set of primitive idempotents, then from Proposition \ref{subring} we know that $\Omega(G)$ or $\Omega_p(G)$ is a subring of the direct product above, therefore also has a finite complete set of primitive idempotents.
\end{proof}
Now we characterize the image of $\phi$ via the property called congruence relation. The most important tool we use is called counting lemma.

\begin{lem}[Counting Lemma]
Let $G$ be a finite group, $X$ a finite $G$-set, and $\cyclic{g}$ the cyclic subgroup of $G$ generated by $g\in G$. Then $\card{G}\cdot \card{X/G}=\sum_{g\in G}\card{X^{\cyclic{g}}}$.

\end{lem}
\begin{proof}
Let $Y=\{(g,x)\in(G,X)\ |\ gx=x\}$, we prove the claimed equation by double counting of the elements of $Y$. Consider the maps
\begin{equation*}
    p:Y\to G: (g,x)\mapsto g,
    \qquad
    q:Y\to X/G: (g,x)\mapsto [x].
\end{equation*}
Note that $p^{-1}(g)={g}\times X^{\cyclic{g}}$, and $q^{-1}([x])=G_x\times Gx$, where $Gx$ is the stabilizer of $x$ and $Gx$ is the orbit of $x$. In particular, $\card{q^{-1}([x])}=\card{G}$. Therefore the left side of the equation is the summation over the fibers of $q$, and the right side is the summation over the fibers of $p$.
\end{proof}

The counting lemma implies the congruence
\begin{equation}
    \sum_{g\in G}\phi_{\cyclic{g}}(X)\equiv 0\modulo\card{G}
\end{equation}
for each finite $G$-set $X$.

Let $NH$ denote the normalizer subgroup of $H$ in $G$, and $WH\coloneqq NH/H$. For finite $G$-set $X$, it is clear that $X^H$ is a finite $WH$-set, and for $nH\in WH$, $(X^H)^{nH}=X^{\cyclic{n,H}}$ where $\cyclic{n,H}$ denotes the subgroup generated by $H$ and $n\in G$. Applying the counting lemma we have
\begin{equation}\label{cong}
    \sum_{nH\in WH}\phi_{\cyclic{n,H}}(X)\equiv 0\modulo \card{WH}.
\end{equation}
\begin{rem}
Another important property of $WH$ is that it acts on $G/H$ freely, therefore for each $(K)\in C(G)$, we have
\begin{equation*}
    (G/H)^{K}\equiv 0\modulo \card{WH}.
\end{equation*}
Therefore we can construct
\begin{equation*}
    x_{(H)}= \frac{1}{\card{WH}}\phi([G/H])\in \tilde{\Omega}(G).
\end{equation*}
Note that as $(G/H)^H=\card{WH}$, the collection of $x_{(H)}$ forms a triangular matrix with units on the diagonal. Therefore $\{x_{(H)}|(H)\in C(G)\}$ forms a basis of $\tilde{\Omega}(G)$. 
\end{rem}
\begin{note}
For $f\in \tilde{\Omega}(G)$, we denote the value of $f$ on the basis element $(H)\in C(G)$ by $f(H)$.
\end{note}
We prove in the next theorem that the congruence relation is the characteristic property of elements in the Burnside ring $\Omega(G)\subseteq \tilde{\Omega}(G)$.
\begin{thm}
An element $f\in \tilde{\Omega}(G)$ is in the subring $\Omega(G)$ if and only if for each $(H)\in C(G)$,
\begin{equation}\label{cong2}
    \sum_{nH\in WH}f(\cyclic{n,H})\equiv 0\modulo \card{WH},
\end{equation}
namely, when the congruence relation in (\ref{cong}) is satisfied
\end{thm}
\begin{proof}
Write $f$ as $f= \sum_{(H)\in C(G)}\alpha(H)x_{H}$, and as in the proof of Proposition \ref{subring}, choose a maximal $(H)$ such that $\alpha(H)\neq 0$, then the congruence relation for $(H)$ means that
\begin{equation*}
    \alpha(H)\equiv 0\modulo\card{WH}.
\end{equation*}
Since $\card{WH}x_{(H)}\in \Omega(G)$, we can remove the summand. Therefore by induction, we prove that $f\in \Omega(G)$.
\end{proof}
\begin{rem}
The summation in (\ref{cong2}) can be rewritten as
\begin{equation}\label{cong3}
    \sum_{(K)}n(H,K)f(K)\equiv 0\modulo \card{WH},
\end{equation}
where the sum is over $G$-conjugacy classes $(K)$ of subgroups $K$, such that $H$ is a normal subgroup of $K$ and $K/H$ is cyclic, including the case $K=H$. Here $n(H,K)$ are integers, and $n(H,H)=1$.
\end{rem}

For prime $p$, let $W_pH$ be a Sylow $p$-subgroup of $WH$, then $X^H$ is a finite $W_pH$-set, then for each subgroup $H$ of $G$ we have a ``$p$-primary'' congruence in the form
\begin{equation}\label{cong4}
    f(H) + \sum_{(K)}l(H,K)f(K)\equiv 0\modulo \card{W_pH}
\end{equation}
where $l(H,K)\in \integer$ and $(H)\leq (K)$.
The ``$p$-primary'' congruence \ref{cong4} as a whole for all prime $p$, which is justified in the next corollary, is as powerful as the congruence in (\ref{cong3}). We will use (\ref{cong3}) and (\ref{cong4}) without worrying too much about the actual summation and the values of the integral coefficients $n(H,K)$ and $l(H,K)$.
\begin{cor}
An element $f\in\tilde{\Omega}(G)$ is in the subring $\Omega(G)$ if and only if for each $(H)\in C(G)$, the congruence relation in (\ref{cong4}) is satisfied for all prime $p$.
\end{cor}
\begin{proof}
Write $f$ as $f=\sum_{(H)\in C(G)}\alpha(H)x_H$, it suffices to show that $\alpha(H)\equiv 0\modulo \card{WH}$.

Consider a maximal $H$ such that $\alpha(H)\neq 0$, then we have $f(H)=\alpha(H)\equiv 0\modulo \card{W_pH}$ for all prime $p$, which means that $\alpha(H)\equiv 0\modulo \card{WH}$. Therefore we can remove this summand, and finish the proof inductively.
\end{proof}

\begin{cor}
An element $f\in \tilde{\Omega}_p(G)$ is in the subring $\tilde{\Omega}_p(G)$ if and only if for each $(H)\in C(G)$, the congruence relation in \ref{cong4} is satisfied for $p$.
\end{cor}
\begin{proof}
    Note that for primes $q\neq p$, the congruence relations satisfy automatically as we can multiply $q$ on both sides of the fraction, so this corollary follows immediately from the last one.
\end{proof}

Now we come back to the idempotents.
\begin{note}
For finite group $G$ and prime $p$, $G$ is called $p$-perfect if $O^p(G)=G$. We denote by $P_p(G)$ the subset of $C(G)$ representing $p$-perfect subgroups of $G$.

For subgroup $H$ of G, we denote by $H_S$ the minimal normal subgroup of $H$ such that $H/H_S$ is solvable. $H$ is perfect if and only if $H_S=H$. We denote by $P(G)$ the subset of $C(G)$ representing perfect subgroups of $G$.
\end{note}

\begin{prop}
An idempotent $f\in\tilde{\Omega}(G)$ is contained in $\Omega(G)$ if and only if for all $(H)\in C(G)$, $f(H)=f(H_S)$.
\end{prop}
\begin{proof}
    Suppose $f\in \Omega(G)$, given subgroup $H$, we have a series
    \begin{equation*}
        H_S=H_0\leq H_1\leq...\leq H_n = H,
    \end{equation*}
    such that $H_{i}$ is normal subgroup of $H_{i+1}$ and $H_{i+1}/H_{i}$ is cyclic of prime order $p(i)$. For subgroup $H$ of $G$, let $f^{H}$ denote the set of fixed elements by $H$ of $f$ viewed as a $G$-set. Then using the counting lemma of $H_{i+1}/H_i$ acting on $f^{H_i}$ gives that $f(H_{i+1})\equiv f(H_i)\mod p(i)$, and since $f(H_K)$ for any $K$ is either $1$ or $0$, we have $f(H_i)=f(H_{i+1})$, and then $f(H)=f(H_S)$.
    
    Conversely, assume $f(H)=f(H_S)$, then we have immediately that $f(H)=f(K)$ for all pairs where $H$ is a normal subgroup of $K$ and $K/H$ is cyclic. Therefore the congruence relation in \ref{cong3} is satisfied for all $H$.
\end{proof}
We can prove analogously the local version of the corollary above, all we need to do is to use the $p$-local version of the congruence relations. The statement of $p$-local primitive idempotents is the next proposition.
\begin{prop}\label{p-local idempotents}
An idempotent $f\in\tilde{\Omega}_p(G)$ is contained in $\Omega_p(G)$ if and only if for all $(H)\in C(G)$, $f(H)=f(O^p(H))$.
\end{prop}

\begin{cor}
The set of primitive idempotents of $\Omega(G)$ corresponds bijectively to $P(G)$. The set of primitive idempotents of $\Omega_p(G)$ corresponds bijectively to $P_p(G)$.
\end{cor}

Now we focus on the expression of primitive idempotents of $\Omega_p(G)$, but the analogous constructions and propositions for $\Omega(G)$ are easy to deduce.

\begin{note}
For finite group $G$, the set of subgroups of $G$ is a locally finite poset. When we write $\sigma$, $\zeta$, and $\mu$, we mean the corresponding function defined for this poset. 

For a $p$-perfect subgroup $H$, we define
\begin{equation*}
    \lambda(D,H)\coloneqq \sum_S\mu(D,S),
\end{equation*}
the summation is over the set of all subgroups $S$ of $G$ such that $O^p(S)=H$. Clearly $\lambda(D,H)\neq 0$ only if $O^p(D)\leq H$ and $D\leq NH$.
\end{note}

For each $p$-perfect subgroup $H$ of $G$, define
\begin{equation*}
    e_p^H\coloneqq \frac{1}{\card{NH}}\sum_{D\leq NH}\card{D}\lambda(D,H)[G/D].
\end{equation*}
\begin{thm}\label{idempotents}
$\{e_p^H\ |\ (H)\in P_p(G)\}$ is the complete set of primitive idempotents of $\Omega_p(G)$.
\end{thm}
\begin{proof}
First, we consider the ring $\Omega_{\rational}(G)$, which can be thought of as the biggest ring, in the sense that every localization of $\Omega(G)$ is a subring of it.

In the case of $\Omega_{\rational}(G)$, the corresponding expression turns to:
\begin{equation*}
     e^H\coloneqq \frac{1}{\card{NH}}\sum_{D\leq H}\card{D}\mu(D,H)[G/D].
\end{equation*}
As $\phi_{\rational}$ is injective as a ring homomorphism, and therefore injective as a linear map between $\rational$ vector-spaces of the same dimension, it is an isomorphism. Therefore the primitive idempotents of $\Omega_{\rational}(G)$ are just $\{e_{(H)}|(H)\in C(G)\}$. Therefore it suffices to prove that
\begin{equation*}
    \begin{aligned}
    \phi_K(e^H)&=1 &\text{   if } (H)=(K)\\
               &=0 &\text{   otherwise}.
    \end{aligned}
\end{equation*}
For $g\in G$ and $K$ subgroup of $G$, denote $g^{-1}Kg$ by $K^g$. Note that we have 
\begin{equation*}
\begin{aligned}
    \phi_K([G/D])&= (G/D)^K\\
                 &= \card{\{g\in G|K^g\in D\}}/\card{D}\\
                 &=\frac{1}{\card{D}}\sum_{g\in G}\zeta(K^g, D),
\end{aligned}
\end{equation*}
so by the \mobius inversion formula
\begin{equation*}
\begin{aligned}
    \phi_K(e^H)&=\frac{1}{\card{NH}}\sum_{g\in G}\sum_{D\leq H}\zeta(K^g,D)\mu(D,H)\\
    &=\frac{1}{\card{NH}}\sum_{g\in G}\sigma(K^g,H)\\
    &=1 \qquad\text{  if } (H)=(K)\\
    &=0 \qquad\text{  otherwise}.
\end{aligned}
\end{equation*}
So we have proved that for each subgroup $H$, $e^H=e_{(H)}\in\Omega_{\rational}(G)$.

Now to come to the $p$-local case. From Proposition \ref{p-local idempotents} we know that the primitive idempotent for $p$-perfect $H$ in $\Omega_p(G)$ is the summation of $\{e_{(S)}|O^p(S)=H\}$. Therefore given a complete set of representatives of the $\{e_{(S)}|(O^p(S)=(H)\}$, denoted by $S_1,S_2,...$, we have
\begin{equation*}
    \begin{aligned}
    \sum_{i}e^{S_i}&=\sum_S\frac{1}{\card{NH:NS}}e^S\\
    &=\sum_S\sum_{D\leq S}\frac{\card{D}}{\card{NH}}\mu(D,S)[G/D]\\
    &=\frac{1}{\card{NH}}\sum_{D\leq NH}\card{D}\lambda(D,H)[G/D]\\
    &=e_p^H,
    \end{aligned}
\end{equation*}
where the summation $\sum_S$ means summation over subgroups $S$ such that $O^p(S)=H$.
This equation shows that $e_p^H$ is exactly the expression of the primitive idempotent corresponding to $(H)\in P_p(G)$.
\end{proof}
\begin{rem}
It seems that we just proved that the expression holds in $\Omega_{\rational}(G)$ instead of $\Omega_p(G)$. However, note that we can decompose the primitive idempotents uniquely in $\Omega_p(G)$ with basis $\{[G/H]\}$, and the coefficients are in $\integer_{(p)}$, but this can also be viewed as the unique decomposition in $\Omega_{\rational}(G)$ with the same basis, so they must coincide. In particular, in the above expression, the coefficient of $[G/D]$ in $e_p^H$ is in $\integer_{(p)}$.
\end{rem}
\newpage

\section{Cardinality of Classifying Space}
In this final chapter, we use the algebraic tools introduced in the last chapter to get expressions for the cardinalities of the classifying space of cyclic group and general finite group. Moreover, we could make one step further to get an expression for the cardinality of any $\pi$-finite space. The deduction of final expressions depends very much on the assumption that \catC\ is of height $1$, but the expression in section $5.2$, which we call the partition of unity, applies in more general cases.
\subsection{Cyclic Group}
First, we have the following expression for the classifying space of a cyclic group, which we could use easily in the following sections.

\begin{prop}\label{cyclic}
Let \catC\ be a $p$-local \inftycat\ of height $1$ with prime $p$. For each finite $p$-group $G$ of order $p^m$, we have
\begin{equation*}
    \card{BG} = \card{BC_p}^m.
\end{equation*}
\end{prop}
\begin{proof}
Recall from group theory that, a $p$-group always has a non-trivial center $H$, therefore $G$ has a normal subgroup isomorphic to $C_p$, so we have a principal fiber sequence
\begin{equation*}
    BC_p\to BG\to B(G/C_p),
\end{equation*}
induced by the sequence $C_p\to G\to G/C_p$. As \catC\ is of height $1$, $\card{C_p}$ is invertible, so by Proposition \ref{principal fiber}, $\card{BG}=\card{C_p}\card{B(G/H)}$. If we assume that $\card{B(G/C_p)} = \card{BC_p}^{m-1}$, then $\card{G} = \card{BC_p}^m$. Therefore by induction on $m\geq 1$, the proof is finished.
\end{proof}

\subsection{General Finite Group}
In this section, we use the expression for the Burnside ring of $G$ to give a ``partition of unity'' in the endomorphisms of the identity in the category of local systems $\Fun(BG, \mathcal{C})$. First, we clarify the notion of ``canonical'' ring structure on the endomorphism unit object.
\begin{prop}
Let \catC\ be a stable symmetric monoidal \inftycat, and let $\one$ be the unit object. Then $\pi_0(\one)=\pi_0(\Map(\one,\one))$ has a canonical ring structure. 
\end{prop}
\begin{proof}
Since \catC\ is stable, it is enriched in abelian groups, therefore $\pi_0(\one)$ has a canonical ring structure, with the composition of maps being the product, and the addition is explicitly given by the semiadditive addition.
\end{proof}
\begin{rem}
From the canonical ring structure of $\pi_0(\one)$, we easily have a induced ring structure of $\pi_0(\one^{BG})$, where $\one^{BG}\coloneqq \one_{\Fun(BG,\mathcal{C})}$.
\end{rem}

\begin{lem}\label{G decomp}
The following diagram is a pullback square:
\begin{equation*}
\begin{tikzcd}
    BK_1\sqcup...\sqcup BK_m
    \arrow{r}\arrow{d}&
    BH_1\arrow{d}\\
    BH_2\arrow{r}&BG,
\end{tikzcd}
\end{equation*}
where $K_1, ..., K_m$ correspond to the decomposition $G/H_1\times G/H_2 = G/K_1\sqcup...\sqcup G/K_m$.
\end{lem}
\begin{proof}
Consider the category of spaces, denoted by $\Spc$, and the following pullback diagram in $\Spc$:
\begin{equation*}
    \begin{tikzcd}
        G/H_1\times G/H_2\arrow{r}\arrow{d}
        &G/H_1\arrow{d}\\
        G/H_2\arrow{r}&\pt,
    \end{tikzcd}
\end{equation*}
taking $G$-quotients, we have the following diagram:
\begin{equation*}
    \begin{tikzcd}
        BK_1\sqcup...\sqcup BK_m
    \arrow{r}\arrow{d}&
    BH_1\arrow{d}\\
    BH_2\arrow{r}&BG.
    \end{tikzcd}
\end{equation*}
Since the functor of taking quotients $(-)/G:\Spc^{BG}\to \Spc$ preserves pullback, the proof is finished.
\end{proof}
The last lemma immediately implies the following homomorphism from a localization of the Burnside ring of $G$ to $\pi_0(\one^{BG})$.
\begin{prop}
Let \catC\ be a stable symmetric  monoidal $p$-local \inftycat\ where $p$ is a prime, and $Q$ a Sylow $p$-subgroup of $G$. Suppose \catC\ is $0$-semiadditive, there is a ring homomorphism
\begin{equation*}
    \psi:\Omega_p(G)[G/Q]^{-1}\longrightarrow
    \pi_0(\one^{BG}).
\end{equation*}
\end{prop}
\begin{proof}
The homomorphism is defined on the additive basis on $\Omega(G)$:
\begin{equation*}
    G/H\mapsto \card{BH\to BG}.
\end{equation*}
Then it suffices to check that it preserves the product. Note that by the above definition
\begin{equation*}
    G/H_1\times G/H_2\mapsto \sum_{i}\card{BK_i\to BG}
    = \card{(\sqcup_{i}BK_i)\to BG}.
\end{equation*}
However, from the pullback square in Lemma \ref{G decomp} and Corollary \ref{pullbackint}, we have
\begin{equation*}
    \card{(\sqcup_{i}BK_i)\to BG} = \card{BH_2\to BG}\circ \card{BH_1\to BG},
\end{equation*}
which shows that the product is preserved.

As \catC\ is $p$-local, $\pi_0(\one^{BG})$ is $p$-local. Moreover, $G/Q$ is mapped to $\card{BQ\to BG}$, which is invertible, as the fiber of $BQ\to BG$ is $G/Q$, which is of order not divisible by $p$, and the evaluation at basepoint detects invertibility. Therefore we can localize the homomorphism to get $\psi$ we want.
\end{proof}
With $\psi$ at hand, an expression of $1\in \Omega_p(G)[G/Q]^{-1}$ gives an expression of $1\in \pi_0(\one^{BG})$, which we call the ``partition of unity'', in analogy to the tool we use frequently in differential geometry.

\begin{thm}\label{thm_1}
For finite group $G$ and stable symmetric monoidal $p$-local \inftycat\ \catC, in $\pi_0(\one^{BG})$ we have a partition of unity
\begin{equation*}
    1 = \sum_{\{Q_{i_1},...,Q_{i_k}\}}(-1)^{k-1}\frac{\card{Q_{i_1}\cap...\cap Q_{i_k}}}{\card{G}}\card{B({Q_{i_1}\cap...\cap Q_{i_k}})\to BG}.
\end{equation*}
In the case where \catC\ is a $1$-semiadditive stable symmetric monoidal $p$-local \inftycat, and $f\in\one^{BG}$, then
\begin{equation*}
    \int_{BG}f = \sum_{\{Q_{i_1},...,Q_{i_k}\}}(-1)^{k-1}\frac{\card{Q_{i_1}\cap...\cap Q_{i_k}}}{\card{G}}\int_{BG}f\circ\card{B({Q_{i_1}\cap...\cap Q_{i_k}})\to BG},
\end{equation*}
where the sum is over non-empty subsets of $\{Q_{1},...,Q_{n}\}$, which is the full set of $p$-Sylow subgroups of $G$.
\end{thm}
\begin{proof}
Recall from Proposition \ref{p-local idempotents} that the idempotents of $\Omega_p(G)$ are in the form $\sum_(S)e_{(S)}$, where the sum is over conjugate classes of $S$, such that $O^p(S)=H$ for a fixed $p$-perfect subgroup $H$. For $Q$ a Sylow $p$-subgroup, since $\phi_{S,p}(Q)=0$ if $S$ is not a subgroup of $Q$, then after inverting $[G/Q]$, in the decomposition of $1\in \Omega_p(G)$ into primitive idempotents of $\Omega_p(G)s$, we only have one summand left, which corresponds to the trivial subgroup denoted also by $1$. In particular, from Theorem \ref{idempotents} we have in $\Omega_p(G)[G/Q]^{-1}$
\begin{equation*}
\begin{aligned}
    1 = e_p^{1} &\coloneqq \frac{1}{\card{G}}\sum_{D\leq G}\card{D}\lambda(D,1)[G/D]\\
    &= \frac{1}{\card{G}}\sum_{D\leq G}\card{D}\sum_{S}\mu(D,S)[G/D],
\end{aligned}
\end{equation*}
where $S$ ranges over all $p$-subgroups of $G$. Note that if $D$ here is not a $p$-subgroup, then $\mu(D,S)$ is $0$ for any $p$-subgroup $S$. Since for fixed $D$, we have
\begin{equation*}
\begin{aligned}
    \sum_S\mu(D,S) &= \sum_{\{Q_{i_1},...,Q_{i_k}\}}(-1)^{k-1}\sum_{D\leq U\leq Q_{i_1}\cap...\cap Q_{i_k}}\mu(D,U),
    \end{aligned}
\end{equation*}
where $\sum_{D\leq U\leq Q_{i_1}\cap...\cap Q_{i_k}}\mu(D,U)$ is $0$ unless $D=Q_{i_1}\cap...\cap Q_{i_k}$. Therefore, when we range $D$ over the $p$-subgroups of $G$, each non-empty subset of $\{Q_1,...,Q_n\}$ appears in non-zero values of $\mu$ exactly once, so we have the that
\begin{equation*}
    1 = \sum_{\{Q_{i_1},...,Q_{i_k}\}}(-1)^{k-1}\frac{\card{Q_{i_1}\cap...\cap Q_{i_k}}}{\card{G}}[G/(Q_{i_1}\cap...\cap Q_{i_k})].
\end{equation*}
Mapping to $\pi_0(\one^{BG})$, that turns into
\begin{equation*}
    1 = \sum_{\{Q_{i_1},...,Q_{i_k}\}}(-1)^{k-1}\frac{\card{Q_{i_1}\cap...\cap Q_{i_k}}}{\card{G}}\card{B({Q_{i_1}\cap...\cap Q_{i_k}})\to BG}.
\end{equation*}
Post-composing with $f$ and integral over $BG$ we have the equation we want.
\end{proof}
\begin{rem}
Note that by Proposition \ref{specified fubini}, let $q$ denote the map $BH\to BG$, then we have
\begin{equation*}
    \int_{BH}q^*f=\int_{BG}\int_qq^*f,
\end{equation*}
then by Proposition \ref{specified homo} we have
\begin{equation*}
\begin{aligned}
   \int_{BG}\int_qq^*f&=\int_{BG}f\circ\int_q\Id\\
   &=\int_{BG}f\circ \card{q}.
\end{aligned}
\end{equation*}
Therefore we may further simplify the equation slightly in this direction.
\end{rem}
\begin{lem}\label{A_H}
For path-connected space $A$ with fundamental group $G$, consider the following pullback square:
\begin{equation*}
    \begin{tikzcd}
        A_H\arrow{r}\arrow{d}&A\arrow{d}\\
        BH\arrow{r}&BG,
    \end{tikzcd}
\end{equation*}
where $A\to BG$ is the map corresponding to the universal cover of $A$, and $BH\to BG$ is associated with subgroup inclusion $H\to G$. The pullback $A_H$ is path-connected with fundamental group $H$, and the higher homotopy groups are the same as $A$.
\end{lem}
\begin{proof}
From the long exact sequence of homotopy groups associated with the homotopy pullback square, we have the long exact sequence
\begin{equation*}
   ...\to \pi_{n+1}BG\to \pi_nA_H\to \pi_nA\times\pi_nBH\to \pi_nBG\to...\to \pi_1BG\to \pi_0A_H\to \pi_0(A\times BH).
\end{equation*}
Therefore we have
\begin{equation*}
    \pi_i(A_H)\cong \pi_iA
\end{equation*}
for $i\geq 2$.
Moreover, consider
\begin{equation*}
    0\to \pi_1A_H\to \pi_1A\times \pi_1BH\to \pi_1BG\to 0,
\end{equation*}
since $\pi_1A\times\pi_1BH\to \pi_1BG$ is just $G\times H\to G:(g,h)\mapsto hg^{-1}$, the kernel $\pi_1A_H$ is clearly isomorphic to $H$, and $A_H\to A$ can then be viewed as the covering map of $A$ associated to $H\hookrightarrow G$.
\end{proof}

With the previous lemma, we immediately have the following corollaries.
\begin{cor}
Consider \catC\ as in Theorem \ref{thm_1}, and $A$ a $\pi$-finite \catC-ambidextrous space. Then we have
\begin{equation*}
    \card{A} = \sum_{\{Q_{i_1},...,Q_{i_k}\}}(-1)^{k-1}\frac{\card{Q_{i_1}\cap...\cap Q_{i_k}}}{\card{G}}\card{A_{Q_{i_1}\cap...\cap Q_{i_k}}},
\end{equation*}
where the sum is over non-empty subsets of $\{Q_{1},...,Q_{n}\}$.
\end{cor}
\begin{proof}
From Lemma \ref{A_H} and Corollary \ref{pullbackint} we know
\begin{equation*}
    \card{A_{Q_{i_1}\cap...\cap Q_{i_k}}\to BG} = \card{A\to BG}\circ \card{B(Q_{i_1}\cap...\cap Q_{i_k})\to BG},
\end{equation*}
so we can rewrite the equation in Theorem \ref{thm_1}, namely
\begin{equation*}
     \int_{BG}\card{A\to BG} = \sum_{\{Q_{i_1},...,Q_{i_k}\}}(-1)^{k-1}\frac{\card{Q_{i_1}\cap...\cap Q_{i_k}}}{\card{G}}\int_{BG}\card{A_{Q_{i_1}\cap...\cap Q_{i_k}}\to BG},
\end{equation*}
which gives the required equation.
\end{proof}
\begin{cor}\label{BG expression}
Consider \catC\ as in Theorem \ref{thm_1}, then for finite group $G$, let $Q_1,...,Q_n$ be the Sylow $p$ subgroups of $G$, we have
\begin{equation*}
    \card{BG} = \sum_{\{Q_{i_1},...,Q_{i_k}\}}(-1)^{k-1}\frac{\card{Q_{i_1}\cap...\cap Q_{i_k}}}{\card{G}}\card{B(Q_{i_1}\cap...\cap Q_{i_k})},
\end{equation*}
where the sum is over non-empty subsets of $\{Q_{1},...,Q_{n}\}$.
\end{cor}
\begin{proof}
Simply apply the last corollary to $A=BG$.
\end{proof}
\begin{rem}
In the last corollary, if we assume in addition that \catC\ is of height $1$, then we can further express $\card{B(Q_{i_1}\cap...\cap Q_{i_k})}$ in terms of $\card{BC_p}$.
\end{rem}
Now we give several examples, with the assumption that \catC\ is of height $1$.

\begin{example}
The easiest example is when $G$ has order prime to $p$, so the Sylow $p$ subgroup is trivial. Therefore we have $\card{BG}=\card{G}^{-1}$.
\end{example}

\begin{example}
Take $p=2$, and $G=S_3$, the permutation group. Then the Sylow $2$ subgroups are $C_2$. There are $3$ Sylow $2$ subgroups, the intersection of any two of them is trivial. Applying Corollary \ref{BG expression} we have
\begin{equation*}
\begin{aligned}
\card{BS_3}&=\frac{2\cdot 3}{6}\card{BC_2}-\frac{1\cdot 3}{6}\card{\pt}+\frac{1}{6}\card{\pt}\\
&=\card{BC_2}-1/3.
\end{aligned}
\end{equation*}
\end{example}

\begin{example}
Take $G=GL_2(F_p)$. The order of $GL_2(F_p)$ is $(p^2-1)(p^2-p)=p(p+1)(p-1)^2$, so the Sylow $p$ subgroup is $C_p$. Note that the matrix
$\begin{pmatrix}
1&1\\
0&1
\end{pmatrix}$
has order $p$, hence it generates a Sylow $p$-subgroup denoted by $Q$. From Sylow theorems, the number of Sylow $p$ subgroups is $\card{G:NQ}$. Note that $NQ=\{\begin{pmatrix}a&b\\0&c\end{pmatrix}|a,c\neq 0\}$, so $\card{NQ}=p(p-1)^2$, therefore the number of Sylow $p$ subgroups is $p+1$.

Applying Corollary \ref{BG expression} we have
\begin{equation*}
    \begin{aligned}
       \card{BGL_2(F_p)}&=\frac{p(p+1)}{p(p+1)(p-1)^2}\card{BC_p}+\sum_{2\leq i\leq (p+1)}{p+1\choose i}(-1)^{i-1}\frac{1}{p(p+1)(p-1)^2}\card{\pt}\\
       &=\frac{1}{(p-1)^2}\card{BC_p}-\frac{1}{(p+1)(p-1)^2}.
    \end{aligned}
\end{equation*}
Here we used the fact that for $n\in \NN$, $\sum_{2\leq i\leq n}(-1)^{i-1}{n\choose i}=1-n$.
\end{example}

\begin{example}\label{example C_q^p}
Consider prime $p$ and let $G$ be the permutation group $S_p$. Then the Sylow $p$ subgroups of $S_p$ are $C_p$. 

The number of elements in $S_p$ of order $p$ is $(p-1)!$, in correspondence to the number of $p$-cycles $(1,a_2,...,a_p)$. The intersection of any two Sylow $p$ subgroups is the trivial subgroup, so the number of Sylow $p$ subgroups is just $(p-2)!$. Applying Corollary \ref{BG expression} we have
\begin{equation*}
\begin{aligned}
   \card{BS_p} &= \frac{p}{p!}\card{BC_p}+\sum_{2\leq i\leq (p-2)!}{(p-2)!\choose i}(-1)^{i-1}\frac{1}{p!}\card{\pt}\\
   &= \frac{1}{(p-1)!}\card{BC_p}+\frac{1-(p-2)!}{p!}.
\end{aligned}
\end{equation*}
Note that by the third Sylow Theorem (see \cite[Theorem~7.7.6]{ArtinAlg}) we have $(p-2)!\equiv 1\modulo p$, therefore we do not have to actually divide $p$ in this calculation. Also, $(p-2)!\equiv 1\modulo p$ can be viewed as a corollary of Wilson's Theorem (see \cite{Wilson}).
\end{example}
\subsection{General Space}
In this section, we give an expression for the cardinality of a general $\pi$-finite space $A$ in an \inftycat\ of height $1$. The procedure consists of several decompositions of $A$.

Throughout this section, unless specified otherwise, we always assume $A$ is $\pi$-finite path connected, and the local systems are \catC-valued, where \catC\ is stable, symmetric monoidal, $p$-local, $1$-semiadditive and of height $1$ (therefore $\infty$-semiadditive). By Remark \ref{WLOGsymmetric} this assumption will not cause loss of generality.
\begin{lem}
Suppose $A$ has fundamental group $G$, then it can be written as
\begin{equation*}
    A \cong A_{q_1}\times_{BG}...\times_{BG}A_{q_n},
\end{equation*}
where $q_i$ are prime numbers, and each $A_{q_i}$ is a path-connected space with fundamental group $G$, and whose higher homotopy groups are $q_i$-groups.
\end{lem}
\begin{proof}
By \cite[Theorem~5.7]{InfinitySylow}, the universal cover of $A$ can be written as
\begin{equation*}
    \tilde{A}\cong \prod_{q}\tilde{A}_q,
\end{equation*}
where $q$ ranges over all prime numbers and $\tilde{A}_q$ is simply connected with $q$-higher homotopy groups. Since $A_{q_i}$ is the $q_i$-completion of $A$, the decomposition is functorial, and thus invariant under $G$. Therefore taking the $G$-quotient we have the required equation.
\end{proof}

Last lemma shows that we could compute $\card{A\to BG}$ by the product of $\card{A_q\to BG}$. Moreover, we could use Proposition \ref{relative principal} to further simplify the calculation with the Postnikov tower of $A$.

First, we give a brief recollection of the Postnikov tower of $A$.

Let $A$ be $\pi$-finite with fundamental group $G$, consider a Postnikov tower for $A$:
\begin{equation*}
    \begin{tikzcd}
        &\vdots\arrow{d}\\
        &A_3\arrow{d}\\
        &A_2\arrow{d}\\
        A\arrow{uur}\arrow{ur}\arrow{r}
        &A_1.
    \end{tikzcd}
\end{equation*}
Now for fibration $A_{i}\to A_{i-1}$, $i\geq 2$, the fiber is $K(\pi_nA,n)=B^n(\pi_nA)$. Consider the universal covering of $A$, denoted by $\tilde{A}$, the corresponding Postnikov tower is principal, and can be thought of as a universal covering of the Postnikov tower of $A$:
\begin{equation*}
    \begin{tikzcd}
        &\vdots\arrow{d}\\
        &\tilde{A}_3\arrow{d}\\
        &\tilde{A}_2\arrow{d}\\
        \tilde{A}\arrow{uur}\arrow{ur}\arrow{r}
        &\tilde{A}_1.
    \end{tikzcd}
\end{equation*}
The fiber of $\tilde{A}_{i}\to \tilde{A}_{i-1}$ is also $B^n(\pi_nA)$, and the action of $G$ on the fiber in induced from the action on $\pi_nA$. In short, the morphisms and fibrations in the Postnikov tower of $\tilde{A}$ are $G$-equivariant. Therefore we have the following lemma.

\begin{lem}\label{card of Eilenberg}
Suppose the fundamental group of space $A$ is $p$-group, denoted by $P$, and suppose that $Q\coloneqq \pi_iA$ is a group of order prime to $p$, then
\begin{equation*}
    \card{(B^iQ)/P\to BP} = (-1)^i\card{Q/P\to BP}.
\end{equation*}
Note that $Q/P$ here means the quotient of a set by a group action, instead of a quotient group.
\end{lem}
\begin{proof}
For $k\geq 1$, we have the principal fibration
\begin{equation*}
    B^{k-1}Q\to \pt\to B^kQ.
\end{equation*}
By Proposition \ref{relative principal}, if in addition, $\card{(B^{k-1}Q)/P\to BP}$ is invertible, then we have
\begin{equation*}
    \card{(B^kQ)/P\to BP}=\card{(B^{k-1}Q)/P\to BP}^{-1}.
\end{equation*}
Therefore if in the base case $\card{(Q/P)\to BP}$ is invertible, then
\begin{equation*}
    \card{(B^kQ)/P\to BP}=\card{(Q/P)\to BP}^{(-1)^k}.
\end{equation*}
To show that $\card{(Q/P)\to BP}$ is invertible, it suffices to show that each fiber is invertible. However, each fiber is just $Q$ as a discrete space, so $\card{Q}$ is invertible because the order of $Q$ is not divisible by $p$.
\end{proof}
With the previous lemma, we could calculate $\card{A\to BP}$ using the Postnikov tower.
\begin{prop}
Consider space $A$ with fundamental group $P$ a $p$-group, and all higher homotopy groups are $q$ groups, as described in the lemma above, then
\begin{equation*}
    \card{A\to BP} = \prod_{i\geq 2}\card{(\pi_iA)/P\to BP}^{(-1)^i}.
\end{equation*}
\end{prop}
\begin{proof}
Consider the Postnikov tower of $A$ with the same notation as above. From Proposition \ref{relative principal}, for $i\geq 2$ we have
\begin{equation*}
    \card{A_i\to BP} = \card{A_{i-1}\to BP}\circ\card{((B^i\pi_iA)/P)\to BP}.
\end{equation*}
Note that $A_1\to BP$ is just $BP\to BP$, combining the expression in Lemma \ref{card of Eilenberg}, we have
\begin{equation*}
    \card{A\to BP} = \prod_{i\geq 2}\card{(\pi_iA)/P\to BP}^{(-1)^i}.
\end{equation*}
\end{proof}
We have considered in the above proposition the case denoted by $A_q$, where $q\neq p$. For the case $q=p$ we have the following proposition. Note that in this case we actually use the property of \catC\ being of height $1$.
\begin{prop}
If $A$ is a $p$-space, then
\begin{equation*}
    \card{A\to BP} = \prod_{i\geq 2}\card{\pi_iA\times BP\to BP}^{-1^{i-1}}
\end{equation*}
\end{prop}
\begin{proof}
    As in Lemma \ref{card of Eilenberg}, we again denote $\pi_iA$ by $Q$, and in this case, $Q$ is a $p$-group, and the Postnikov tower of $A$ can be refined so as to have the fibers with trivial action. Therefore we have
    \begin{equation*}
        (B^iQ)/P\cong B^iQ\times BP.
    \end{equation*}
    
    As in the proof of Lemma \ref{card of Eilenberg}, since the fiber of $BQ\times BP \to BP$ is just $BP$, which is invertible when \catC\ is of height $1$, $\card{BQ\times BP \to BP}$ is invertible. Therefore we have
    \begin{equation*}
        \card{(B^iQ/P)\to BG} = \card{(BQ/P)\to BG}^{-1^{i-1}}.
    \end{equation*}
    In particular, $\card{(B^iQ/P)\to BG}$ is invertible, therefore from Proposition \ref{relative principal} we have
    \begin{equation*}
        \card{A\to BP} = \prod_{i\geq 2}\card{B\pi_iA\times BP\to BP}^{-1^{i-1}}.
    \end{equation*}
    
\end{proof}

\begin{rem}
Note that $\card{B\pi_iA\times BP\to BP}$ is canonically determined by $\card{B\pi_iA}$, therefore we could slightly abuse the notation, and write $\card{B\pi_iA\times BP\to BP}$ as $\card{B\pi_iA}_{BP}$.
\end{rem}

\begin{note}
For an abelian group $G$ and prime $q$, we denote by $G_q$ for the Sylow $q$-subgroup.
\end{note}
Finally, in the next two corollaries, we combine what we have, and build an expression of $\card{A}$.
\begin{cor}
If $A$ is a space with fundamental group $P$ a $p$-group, and \catC\ is of height $1$. then
\begin{equation*}
    \card{A\to BP} = \prod_{i\geq 2}\card{(B(\pi_iA)_p)}_{BP}^{-1^{i-1}}\prod_{q\neq p}\card{((\pi_iA)_q/P\to BP}^{-1^i}.
\end{equation*}
\end{cor}

\begin{cor}\label{final equation}
Let $A$ be a space with fundamental group $G$, and let $Q_1,...,Q_n$ be the Sylow $p$ subgroups of $G$, then we have
\begin{equation*}
    \begin{aligned}
    \card{A} &= \sum_P(-1)^{k-1}\frac{\card{P}}{\card{G}}\card{A_P}\\
         &= \sum_P(-1)^{k-1}\frac{\card{P}}{\card{G}}\int_{BP}\prod_{i\geq 2}\card{(B(\pi_iA)_p)}_{BP}^{-1^{i-1}}\prod_{q\neq p}\card{((\pi_iA)_q/P\to BP}^{-1^i}\\
         &= \sum_P(-1)^{k-1}\frac{\card{P}}{\card{G}}\left(\prod_{i\geq 2}\card{B(\pi_iA)_p}^{-1^{i-1}}\right)\int_{BP}\left(\prod_{i\geq 2}\prod_{q\neq p}\card{((\pi_iA)_q/P\to BP}^{-1^i}\right)\\
         &= \left(\prod_{i\geq 2}\card{B(\pi_iA)_p}^{-1^{i-1}}\right)\sum_P(-1)^{k-1}\frac{\card{P}}{\card{G}}\int_{BP}\left(\prod_{i\geq 2}\prod_{q\neq p}\card{((\pi_iA)_q/P\to BP}^{-1^i}\right)\\
         &=\left(\prod_{i\leq 2}\card{BC_p}^{-1^{i-1}\log_{p}\card{\pi_i(A)_p}}\right)\sum_P(-1)^{k-1}\frac{\card{P}}{\card{G}}\int_{BP}\left(\prod_{i\geq 2}\prod_{q\neq p}\card{((\pi_iA)_q/P\to BP}^{-1^i}\right)
         \end{aligned}
\end{equation*}
where $P$ range over the intersection of the elements of non-empty subsets of $\{Q_1,...,Q_n\}$.
\end{cor}

\begin{rem}
Note that $(\pi_iA)_q/P\to BP$ can be regarded as $BH_1\sqcup...\sqcup BH_m$, corresponding to the decomposition $(\pi_iA)_q=P/H_1\sqcup...\sqcup P/H_m$, which is just $(\pi_iA)_q$ inside the Burnside ring.  Note also that if we split the summands of the product $\prod_{i\geq 2}\prod_{q\neq p}\card{((\pi_iA)_q/P\to BP}^{-1^i}$ into two subsets with degree $1$ and $-1$ respectively, we can do multiplications effectively for each subset in the Burnside ring $\Omega_p(P)$, which makes the integration much easier. However, the calculation in this expression is still very strenuous and should be rather considered to be the job of some algebraic programming algorithm.
\end{rem}

\begin{example}
Consider the permutation group $S_p$ and the permutation action of $S_p$ on $C_q^p$. This induces a free action of $S_p$ on $B^2(C_q^p)$, so taking the quotient we have space $A$, with fundamental group $S_p$ acting on $\pi_2(A)=C_q^p$ by the permutation action. Then from Corollary \ref{final equation} and Example \ref{example C_q^p} we have
\begin{equation*}
    \begin{aligned}
       \card{A} &= \frac{1}{p-1}\card{A_{C_p}}+\frac{1-(p-2)!}{p!}A_{\pt} \\
       &= \frac{1}{p-1}\card{A_{C_p}}+\frac{1-(p-2)!}{p!}B^2(C_q^p)\\
       &= \frac{q}{p-1}\card{BC_p}+\frac{q^p-q}{p(p-1)}+\frac{1-(p-2)!}{p_!}\cdot q^p.
    \end{aligned}
\end{equation*}
In the last step we just decompose the group $C_q^p$ as quotients of $C_p$, namely
\begin{equation*}
    C_q^p = q\cdot C_p/C_p + \frac{q^p-q}{p}C_p/1.
\end{equation*}
\end{example}
\newpage

\bibliographystyle{alpha}
\phantomsection\addcontentsline{toc}{section}{\refname}
\bibliography{cardiref}

\begin{thebibliography}{GGN16}

\bibitem[AR08]{ausoni2008chromatic}
Christian Ausoni and John Rognes.
\newblock {The chromatic red-shift in algebraic K-theory}.
\newblock {\em L`Enseignement Math{\'e}matique}, 54(2):13--15, 2008.

\bibitem[Art11]{ArtinAlg}
Michael Artin.
\newblock {\em Algebra}, volume 543.
\newblock Pearson Education., 2011.

\bibitem[Bec]{DualHist}
James~C. Becker.
\newblock Representation theorya history of duality in algebraic topology.
\newblock {https://www.math.purdue.edu/~gottlieb/Bibliography/}.

\bibitem[BK72]{CoreRing}
A.K. Bousfield and D.M. Kan.
\newblock The core of a ring.
\newblock {\em Journal of Pure and Applied Algebra}, 2(1):73--81, 1972.

\bibitem[CSY18]{TeleAmbi}
Shachar Carmeli, Tomer~M Schlank, and Lior Yanovski.
\newblock Ambidexterity in chromatic homotopy theory.
\newblock {\em arXiv preprint arXiv:1811.02057}, 2018.

\bibitem[CSY20]{AmbiHeight}
Shachar Carmeli, Tomer~M Schlank, and Lior Yanovski.
\newblock {Ambidexterity and Height}.
\newblock {\em arXiv preprint arXiv:2007.13089}, 2020.

\bibitem[Die]{RepTheory}
Tammo~tom Dieck.
\newblock Representation theory.
\newblock {https://www.uni-math.gwdg.de/tammo/}.

\bibitem[GGN16]{GepnerUniversality}
David Gepner, Moritz Groth, and Thomas Nikolaus.
\newblock Universality of multiplicative infinite loop space machines.
\newblock {\em Algebraic \& Geometric Topology}, 15(6):3107--3153, 2016.

\bibitem[Har17]{ambiSpan}
Yonatan Harpaz.
\newblock Ambidexterity and the universality of finite spans.
\newblock 2017.

\bibitem[Hat02]{HatcherTop}
Allen Hatcher.
\newblock {\em Algebraic Topology}, volume 544.
\newblock Cambridge University Press., 2002.

\bibitem[HL13]{AmbiKn}
Michael Hopkins and Jacob Lurie.
\newblock {Ambidexterity in $K(n)$-local stable homotopy theory}.
\newblock {\em preprint}, 2013.

\bibitem[HS99]{hoveymorava}
Mark Hovey and Neil~P Strickland.
\newblock {\em Morava $K$-theories and localisation}, volume 666.
\newblock American Mathematical Soc., 1999.

\bibitem[Lura]{LurieChrome}
Jacob Lurie.
\newblock Chromatic homotopy theory.
\newblock {http://www.math.harvard.edu/~lurie/}.

\bibitem[Lurb]{HA}
Jacob Lurie.
\newblock Higher algebra.
\newblock {http://www.math.harvard.edu/~lurie/}.

\bibitem[Lur09a]{htt}
Jacob Lurie.
\newblock {\em Higher topos theory}, volume 170 of {\em Annals of Mathematics
  Studies}.
\newblock Princeton University Press, Princeton, NJ, 2009.

\bibitem[Lur09b]{TopFieldTheory}
Jacob Lurie.
\newblock {On the Classification of Topological Field Theories}.
\newblock {\em arXiv preprint arXiv:0905.0465}, 2009.

\bibitem[PS17]{InfinitySylow}
Matan Prasma and Tomer~M. Schlank.
\newblock Sylow theorems for $\infty$-groups.
\newblock {\em Topology and its Applications}, 222:121--138, 2017.

\bibitem[Rot64]{Rota}
Gian-Carlo Rota.
\newblock On the foundations of combinatorial theory i. theory of m{\"o}bius
  functions.
\newblock {\em Zeitschrift f{\"u}r Wahrscheinlichkeitstheorie und verwandte
  Gebiete}, 2(4):340--368, 1964.

\bibitem[Wei]{Wilson}
Eric~W. Weisstein.
\newblock Wilson's theorem. {From MathWorld---A Wolfram Web Resource}.
\newblock {https://mathworld.wolfram.com/WilsonsTheorem.html}.

\bibitem[Yos83]{BurnsideIdem}
Tomoyuki Yoshida.
\newblock Idempotents of burnside rings and dress induction theorem.
\newblock {\em Journal of Algebra}, 80:90--105, 1983.

\end{thebibliography}

\end{document}